\documentclass[10pt,reqno]{amsart}
\usepackage{amssymb,amsmath,bm,geometry,graphics,url,color}
\usepackage{amsfonts}
\usepackage{graphicx}
\usepackage{epstopdf}
\usepackage{pdfpages}
\usepackage{epsfig,multirow}
\usepackage{amscd}
\usepackage{enumerate}
\usepackage{mathrsfs}
\usepackage{xypic}
\usepackage{stmaryrd}
\usepackage{fancybox}
\usepackage{exscale,array}
\usepackage{enumitem}%

\def\pdt2{\partial_t^2}
\def\pdx2{\partial_x^2}

\newcommand{\normmm}[1]{{\left\vert\kern-0.25ex\left\vert\kern-0.25ex\left\vert #1
    \right\vert\kern-0.25ex\right\vert\kern-0.25ex\right\vert}}

\newcommand{\abs}[1]{\left\vert#1\right\vert}

\def\RR{{\mathbb{R}}}

\def\bigo{{\mathcal O}}

\def\d{{\mathrm d}}
\def\e{{\mathrm e}}

\def\iu{\mathrm{i}}
\def\eps{\varepsilon}

\DeclareMathOperator{\sinc}{sinc}

\newtheorem{theo}{Theorem}[section]
\newtheorem{lem}[theo]{Lemma}

\newtheorem{rem}[theo]{Remark}

\newtheorem{prop}[theo]{Proposition}
\newtheorem{algorithm}[theo]{Algorithm}
\numberwithin{equation}{section}

\title[Error estimates for charged-particle
dynamics]{Error estimates of some splitting schemes for charged-particle
dynamics under strong magnetic field}
\author[B. Wang]{Bin Wang}\address{\hspace*{-12pt}B.~Wang: School of Mathematics and Statistics, Xi'an Jiaotong University, 710049 Xi'an, China}
\email{wangbinmaths@xjtu.edu.cn}\urladdr{http://gr.xjtu.edu.cn/web/wangbinmaths/home}

\author[X. Zhao]{Xiaofei Zhao}
\address{\hspace*{-12pt}X.~Zhao: School of Mathematics and Statistics \& Computational Sciences Hubei Key Laboratory, Wuhan University, 430072 Wuhan, China}
\email{matzhxf@whu.edu.cn}\urladdr{http://jszy.whu.edu.cn/zhaoxiaofei/en/index.htm}
\begin{document}
\maketitle
\dedicatory{}

\begin{abstract}
In this work, we consider the error estimates of some splitting schemes for the charged-particle dynamics under a strong magnetic field. We first propose a novel energy-preserving splitting scheme with computational cost per step independent from the strength of the magnetic field. Then under the maximal ordering scaling case, we establish for the scheme and in fact for a class of Lie-Trotter type splitting schemes, a uniform (in the strength of the magnetic field) and optimal error bound in the position and in the velocity parallel to the magnetic field. For the general strong magnetic field case, the modulated Fourier expansions of the exact and the numerical solutions are constructed to obtain a convergence result. Numerical experiments are presented to illustrate the error and energy behaviour of the splitting schemes.
 \\ \\
{\bf Keywords:} Charged particle dynamics, Strong magnetic field, Splitting scheme, Energy-preserving, Error estimate, Modulated Fourier expansion. \\ \\
{\bf AMS Subject Classification:} 65L05, 65L20, 65L70, 65P10, 78A35, 78M25.
\end{abstract}

\section{Introduction}
The dynamics of charged particles in external electromagnetic field are of fundamental importance in plasma physics.
In this work, we are concerned with the numerical solution of the following charged-particle
dynamics (CPD) under a strong magnetic field \cite{Hairer2018,lubich19}
\begin{equation}\label{charged-particle sts}
\begin{split}
&\dot{x}(t)=v(t),\\
&\dot{v} (t)=v(t)\times \frac{B(x(t))}{\eps}+E(x(t)), \quad t>0,\\
&x(0)=x_0,\quad v(0)=v_0,
\end{split}
\end{equation}
where $x(t):[0,\infty)\to \RR^3$ and $v(t):[0,\infty)\to \RR^3$ are respectively the unknown position and velocity of the particle, $x_0$ and $v_0\in\RR^3$  are the given initial values, $E(x)=-\nabla U(x)$ is a given electric field generated by some scalar potential $U(x)$, $B(x)$ is a given magnetic field and  $\eps\in(0,1]$ is a dimensionless parameter inversely proportional to the strength of the magnetic field. Along the solution of \eqref{charged-particle
sts},  the energy or Hamiltonian $H(t)$ of the system
  \begin{equation}\label{H(x,v)}
H\left(x(t),v(t)\right):=\frac{1}{2}\abs{v(t)}^2+U(x(t))\equiv H\left(x(0),v(0)\right),\quad t\geq0,
\end{equation}
 is conserved.

The CPD has been studied for long times in the physical literature  \cite{Arnold97,Benettin94,Cary2009,add1,Northrop63}. The strong external magnetic field is introduced in important applications such as the magnetic fusion, where such magnetic field is essential for controlling the dynamics of plasma in the tokamak device for fusion. This has attracted many recent modeling and simulation works, and (\ref{charged-particle sts}) frequently occurs as  a core problem to solve after particle discretization of some kinetic models \cite{VP2, CPC,Zhao,VP3, VP4,VP5,VP-filbet,VP6,VP8,VP7,SonnendruckerBook}.

Along the numerical aspect for (\ref{charged-particle sts}), various schemes have been considered in the past decades. Earlier studies have been devoted to address the regime $\eps=1$ in (\ref{charged-particle sts}). Among them, the Boris method \cite{Boris1970} proposed in 1970 is still widely used by physicists, followed by some recent numerical analysis work \cite{Hairer2017-1,Qin2013} to address its mathematical property. Later on, many other structural-preserving   schemes have been designed, including the volume-preserving algorithm \cite{He2015}, the time-symmetric algorithm \cite{Hairer2017-2}, the symplectic or K-symplectic algorithms \cite{He2017,PRL1,Tao2016,Webb2014,Zhang2016}, the Poisson integrators \cite{Ostermann15} and the energy-preserving algorithms \cite{L. Brugnano2019,Li-ANM,Li-AML}.

Recent numerical efforts have been focused on the strong magnetic field regime of CPD, i.e. $0<\eps\ll1$ in (\ref{charged-particle sts}). In \cite{Hairer2018}, the long time near-conservation property of a variational integrator was analyzed for (\ref{charged-particle sts}) under $0<\eps\ll1$.
An exponential energy-preserving integrator was developed in \cite{Wang2020} for (\ref{charged-particle sts}) under a constant strong magnetic field $B$. A filtered Boris algorithm was formulated in \cite{lubich19} under the \emph{maximal ordering scaling} \cite{scaling1,scaling2}, i.e. $B=B(\eps x)$ in (\ref{charged-particle sts}) with $|B(0)|>0$ independent of $\eps$, which improves the asymptotic behaviour of the original Boris method as $\eps\to0$. At the kinetic level, in corporation with the Particle-in-Cell discretization, some more multiscale schemes have been proposed for (\ref{charged-particle sts}) including the asymptotic preserving schemes \cite{VP4,VP5} and the uniformly accurate schemes \cite{VP1,Zhao}. Although these powerful numerical methods have already been proposed,  error estimate results towards (\ref{charged-particle sts}) in the strong magnetic field regime are still limited in the literature to our best knowledge. In particular, even for some standard numerical methods, the optimal dependence of the error of on the step size and $\eps$ is not yet established rigorously. The very recent work \cite{VP9} has done the analysis for the IMEX finite difference scheme.

In this work, we consider the class of splitting type scheme which is undoubtedly one of the most popular classical methods \cite{Splitting} for (\ref{charged-particle sts}), and we aim to analyze its optimal convergence result. On one hand, we first propose a novel energy-preserving splitting scheme for solving the CPD (\ref{charged-particle sts}), where we combine the idea of the average vector field \cite{AVF} and splitting. The scheme exactly preserves the energy (\ref{H(x,v)}) at the discrete level for all times. More importantly, in the scheme the stiffness is not involved in the nonlinear equation thanks to splitting, and so the nonlinear solver can perform efficiently for all $\eps\in(0,1]$. In contrast, the other energy-preserving schemes
such as  the direct average vector field method \cite{AVF}, energy-preserving collocation methods \cite{Hairer000},  energy-conserving line integral methods \cite{L. Brugnano2019} and those from  \cite{Li-ANM,Li-AML} quickly lose efficiency as $\eps$ decreases because of the stiffness in the nonlinear equation.
 On the other hand, under the maximal ordering scaling case of (\ref{charged-particle sts}), we shall for the first time establish the rigorous optimal convergence result for a class of Lie-Trotter type splitting schemes including the proposed energy-preserving splitting and a volume-preserving splitting from the literature \cite{Zhao}. We prove by using the averaging technique \cite{Chartier}, that the schemes exhibit \emph{uniform first order} error bound in $x$ and $v_{\parallel}$ (the component of $v$ parallel to $B$) for $\eps\in(0,1]$,  %when the step size is small enough,
 which seems not true at the first glance of (\ref{charged-particle sts}) due to the $\bigo(1/\eps)$ commutator. % we prove this optimal convergence  result in the maximal ordering scaling case of (\ref{charged-particle sts}).
 For the general strong magnetic field case of (\ref{charged-particle sts}), due to technical difficulty to obtain the stability of the scheme under standard energy approach, we turn to another powerful tool namely the modulated Fourier expansion \cite{ICM,Hairer00,Hairer16,hairer2006}. We shall construct the modulated Fourier expansions of the exact solution and the numerical solution, and then establish a convergence result of the scheme in $\eps$. Numerical results are presented in the end to underline the performance of the schemes. %{\color{blue}To our knowledge, this is the first work
%that investigates two properties (optimal convergence and energy conservation) of splitting schemes for charged-particle
%dynamics under strong magnetic fields.}

The rest of the paper is organized as follows. In section \ref{sec:method}, we propose the energy-preserving splitting scheme. In section \ref{sec:analysis1}, we give the optimal convergence result and the rigorous proof in the maximal ordering scaling case. In section \ref{sec:analysis2}, we carry out the modulated Fourier expansion in the general case and establish the convergence result. The numerical results are given in section \ref{sec:num} and the conclusion is  drawn in section \ref{sec:con}.

\section{Numerical methods}\label{sec:method}
In this section, we shall present the class of splitting schemes.
We shall denote $h=\Delta t>0$ as the time step and $t_n=nh$ for $n\in\mathbb{N}$.
%\subsection{A class of splitting schemes}

Firstly, we introduce the energy-preserving schemes.
The schemes are based on the splitting of \eqref{charged-particle sts} into two following subflows:
\begin{equation}\label{charged-sts-first order}
 \frac{d}{dt }\begin{pmatrix}
    x \\
    v \\
\end{pmatrix}= \begin{pmatrix}
    0 \\
    \frac{1}{\eps}v\times B(x) \\
  \end{pmatrix}
,\qquad  \frac{d}{dt }\begin{pmatrix}
    x \\
    v \\
  \end{pmatrix}
= \begin{pmatrix}
    v \\
    E(x) \\
  \end{pmatrix}.
\end{equation}
For the
first flow, since $x(t)\equiv const$, we have the exact integration for $v$ and so we get the exact propagator
\begin{equation}\label{LM}\Phi^{L}_{t}:\ \ \left(
  \begin{array}{c}
    x(t) \\
    v(t) \\
  \end{array}
\right)=\left(
  \begin{array}{c}
    x(0) \\
    \e^{\frac{t}{\eps}\widehat B (x(0))}v(0) \\
  \end{array}
\right),\quad t\geq0,
\end{equation}
where the skew symmetric matrix $\widehat B $ is given by
$$\widehat B (x)= \begin{pmatrix}
                     0 & b_3(x) & -b_2(x) \\
                     -b_3(x) & 0 & b_1(x) \\
                     b_2(x) & -b_1(x) & 0 \\
                  \end{pmatrix}
$$
with the magnetic field $B=(b_1,b_2,b_3)^\intercal\in \RR^3$. By the Rodrigues type formula \cite{Zhao,lubich19,VP7}, the matrix exponential function $\e^{t\widehat{B}}$ can be efficiently implemented in practice.

The second flow in the splitting (\ref{charged-sts-first order}) is nonlinear, and so we look for approximations. Note it is a canonical Hamiltonian system: $\dot{q}(t)=J^{-1}\nabla H(q(t))$ with $J$ the symplectic matrix, so in order to get the exact energy-preserving property, we adopt
the average vector field (AVF) formula \cite{AVF} which by denoting $q^n\approx q(t_n)$ is defined as
\begin{equation}
 q^{n+1}=q^n+h\int_{0}^{1}J^{-1}\nabla H\left((1-\rho)q^{n}+\rho
 q^{n+1}\right)d\rho,
\label{AVF}%
\end{equation}
and we end up with the following energy-preserving splitting methods.

\begin{algorithm}[Energy-preserving splitting method]
\label{alg:EPS}
For the second flow in (\ref{charged-sts-first order}), we apply the AVF
method \eqref{AVF} to get the  approximated propagator $\Phi^{NL}_{t}$, which reads \begin{equation}\label{NLM}\Phi^{NL}_{t}:\ \
\begin{pmatrix}
    x(t) \\
    v(t) \\
  \end{pmatrix}
=\begin{pmatrix}
    x(0)+tv(0)+\frac{t^2}{2}\int_{0}^1
E\left(\rho x(0)+(1-\rho)x(t)\right) d\rho \\
    v(0)+t\int_{0}^1
E\left(\rho x(0)+(1-\rho)x(t)\right) d\rho  \\
  \end{pmatrix}.
\end{equation}
Then the full scheme can be obtained through composition. For example, by denoting
the numerical solution $x^n\approx x(t_n),\, v^n\approx v(t_n)$ and choosing $x^0=x_0,\,v^0=v_0$,
 the
 Lie-Trotter splitting scheme
\begin{equation*}\Phi_{h}=\Phi^{NL}_{h}\circ\Phi^{L}_{h},\end{equation*}
for solving (\ref{charged-particle sts}) in total reads for $n\geq0$,
\begin{equation}\label{TSM0}
\left\{\begin{split}x^{n+1}=&x^{n}+
h\e^{\frac{h}{\eps}\widehat B (x^n)}v^{n}+\frac{h^2}{2}\int_{0}^1
E\left(\rho x^{n}+(1-\rho)x^{n+1}\right) d\rho,\\
v^{n+1}=&\e^{\frac{h}{\eps}\widehat B (x^n)}v^{n}+h\int_{0}^1
E\left(\rho x^{n}+(1-\rho)x^{n+1}\right) d\rho.
\end{split}\right.
\end{equation}
We shall refer to this algorithm by S1-AVF.
\end{algorithm}

It is noted that Algorithm \ref{alg:EPS} is implicit, while the nonlinear equation (\ref{NLM}) is independent of $\eps$. Therefore, compared with other implicit energy-preserving schemes \cite{L. Brugnano2019,Hairer000,Li-ANM,Li-AML, AVF} for solving CPD (\ref{charged-particle sts}), the computational cost of S1-AVF per time step is uniform in $\eps\in(0,1]$. To obtain an explicit scheme, we  consider the following approximation.

\begin{algorithm}[Explicit splitting method]
\label{alg:ES} For the second flow in (\ref{charged-sts-first order}), we linearize (\ref{NLM}) and now $\Phi^{NL}_{h}$ is
given by
$$\Phi^{NL}_{t}:\ \ \begin{pmatrix}
    x(t) \\
    v(t) \end{pmatrix}
 =\begin{pmatrix}
    x(t)+tv(t)+\frac{t^2}{2}E(x(0)) \\
    v(t)+\frac{h}{2}\left[E(x(0))+E(x(t))\right]
  \end{pmatrix}.
$$
With the same  $\Phi^{L}_{h}$ defined by \eqref{LM}, the Lie-Trotter splitting  yields the scheme: for $n\geq0$,
\begin{equation}\label{TSM1}
\left\{\begin{split}x^{n+1}=&x^{n}+
h\e^{\frac{h}{\eps}\widehat B (x^n)}v^{n}+\frac{h^2}{2}E(x^{n}),\\
v^{n+1}=&\e^{\frac{h}{\eps}\widehat B (x^n)}v^{n}+\frac{h}{2}\left[E(x^{n})+E(x^{n+1})\right],
\end{split}\right.
\end{equation}
for solving (\ref{charged-particle sts}), and we shall refer to it as S1-SV.
\end{algorithm}

For the above two presented algorithms, their energy conservation properties are stated as follows.

\begin{prop}
The Algorithm \ref{alg:EPS}  exactly preserves  the energy \eqref{H(x,v)} at the discrete level, i.e. for $n\in\mathbb{N}$, $H(x^{n},v^{n})\equiv H(x^{0},v^{0}). $
\end{prop}
\begin{proof}%For brevity, we prove this result for S1-AVF.
Denote in S1-AVF (\ref{TSM0}) $$ \begin{pmatrix}
    x^{L} \\
    v^{L}
  \end{pmatrix} = \Phi^{L}_{h}  \begin{pmatrix}
    x^{0} \\
    v^{0}
  \end{pmatrix},\qquad  \begin{pmatrix}
    x^{1} \\
    v^{1}
  \end{pmatrix}= \Phi^{NL}_{h}  \begin{pmatrix}
    x^{L} \\
    v^{L}
  \end{pmatrix}.$$
 Firstly, since $\widehat B (x)$ is skew symmetric, the propagator  $\Phi^{L}_h$ exactly preserves the energy $\frac{1}{2}\abs{v}^2$, i.e.
$\frac{1}{2}\abs{v^{L}}^2=\frac{1}{2}\abs{v^0}^2$, and $x^L=x^0$.
On the other hand for $\Phi^{NL}_h$, it is clearly that $
\frac{d}{dt }\begin{pmatrix}
    x \\
    v
  \end{pmatrix}= \begin{pmatrix}
    v \\
   E(x)
  \end{pmatrix}
$ is a Hamiltonian system
%{\color{blue}$$\frac{d}{dt }\begin{pmatrix}
%    x \\
%    v
%  \end{pmatrix}=\left(
%                  \begin{array}{cc}
%                    0 & -I \\
%                    I & 0 \\
%                  \end{array}
%                \right)^{-1}
%  \nabla \tilde{H}(x,v)$$
with energy
$ \tilde{H}(x,v)=\frac{1}{2}\abs{v}^2+U(x)$.
 Concerning the energy conservation of AVF formula (\ref{AVF}) for such flow, which was established in \cite{Quispel2008}, we obtain in $\Phi^{NL}_h$
$$\frac{1}{2}\abs{v^{1}}^2+U\left(x^{1}\right)=\frac{1}{2}\abs{v^{L}}^2+U\left(x^{L}\right).$$
On the basis of  these results, we have
\begin{equation*}
\begin{aligned}
&H\left(x^{1},v^{1}\right)=\frac{1}{2}\abs{v^{L}}^2+U\left(x^{L}\right)
=\frac{1}{2}\abs{v^{0}}^2+U(x^{0})=H\left(x^{0},v^{0}\right),
\end{aligned}
\end{equation*}
which shows the result for S1-AVF.

By the above fact, the energy conservation of Algorithm \ref{alg:EPS} is straightforward through arbitrary composition.
\end{proof}

It is clear from above that one can switch to other energy-preserving techniques for approximating the nonlinear flow to define $\Phi^{NL}_{h}$, and the algorithm \ref{alg:EPS} is still energy-preserving. A direct result is that when electric field $E(x)$ in (\ref{charged-particle sts}) is constant in space, then we have the preserving property in the explicit scheme.
%Otherwise,
%$\Phi_{h}$ cannot preserve the energy exactly such as algorithm
%\ref{alg:ES}. %For the long time energy conservation behaviour of
%algorithm \ref{alg:ES}, we just need to study this point for
%$\Phi^{NL}_{h}$ applied to $
%\begin{array}[c]{ll}
%\frac{d}{dt }\left(
%  \begin{array}{c}
%    x \\
%    v \\
%  \end{array}
%\right)= \left(
%  \begin{array}{c}
%    v \\
%    F(x) \\
%  \end{array}
%\right)
%\end{array}
%$, which can be derived by  backward error analysis.

\begin{prop}\label{epsv} The Algorithm \ref{alg:ES} preserves the energy (\ref{H(x,v)}) if the external electric field in the CPD (\ref{charged-particle sts}) is a constant field.
\end{prop}

Note the presented way of splitting (\ref{charged-sts-first order}) is different from the one in the literature \cite{Zhao}:
\begin{equation}\label{charged-sts-first order-v}
\frac{d}{dt }\begin{pmatrix}
    x \\
    v
  \end{pmatrix} =\begin{pmatrix}
    v \\
   0
  \end{pmatrix},\qquad \frac{d}{dt }\begin{pmatrix}
    x \\
    v
  \end{pmatrix}= \begin{pmatrix}
    0 \\
    \frac{1}{\eps}v\times B(x)+ E(x)
  \end{pmatrix}.
\end{equation}
where both subflows have exact integrators, and it in combine leads to the following volume-preserving algorithm.
\begin{algorithm}[Volume-preserving splitting method]
\label{alg:VPS} By integrating (\ref{charged-sts-first order-v}) exactly, the Lie-Trotter splitting method for solving (\ref{charged-particle sts}) reads
\begin{equation}\label{vps1}
\left\{\begin{split}x^{n+1}=&x^{n}+
hv^{n+1},\\
v^{n+1}=&\e^{\frac{h}{\eps}\widehat B (x^n)}v^{n}+h\varphi_1\left(\frac{h}{\eps}\widehat B (x^n)\right) E(x^{n}),
\end{split}\right.
\end{equation}
with $\varphi_1(z)=(\e^z-1)/z$ and we  denote it by S1-VP.
\end{algorithm}

The presented three splitting algorithms, i.e. (\ref{TSM0}), (\ref{TSM1}) and (\ref{vps1}) look rather close. In particular, they share the same `linear' part which plays the key role in coming analysis. The main observation of the paper is that all of them show uniform error bound $\bigo(h)$ in the position $x(t)$ and in one component of the velocity $v(t)$ when $h$ is small. This will be illustrated by numerical experiments in section \ref{sec:num}. Such convergence result  seems surprising at the first glance of (\ref{charged-particle sts}), since usually the error of splitting scheme is determined by the commutator which is $\bigo(1/\eps)$ here.  For higher order compositions such as Strang splitting, such uniform error bound is gone. Therefore, in this paper we focus on the three Lie-Trotter type schemes and aim to understand their uniform error bound. The next two sections are devoted to the rigorous error analysis.

\section{Optimal convergence in maximal ordering case}\label{sec:analysis1}
In this section, we give the convergence result of the presented splitting schemes. To get rigorous optimal error estimates, we restrict ourself to first consider the so-called \emph{maximal ordering scaling} \cite{scaling1,lubich19,scaling2} of the CPD (\ref{charged-particle sts}) here, i.e.
\begin{equation}\label{model0}
%\left\{
%\begin{split}
\dot{x}= v,\quad
\dot{v}=\frac{1}{\eps}v\times B(\eps x)+E(x),\quad 0<t\leq T,
%&x(0)=x_0,\quad v(0)=v_0.
%\end{split}\right.
\end{equation}
where the magnetic field $B(\eps x)$ satisfies the condition
$|B(0)|>0$ independent of $\eps$. For simplicity of notations, we shall denote $A\lesssim B$ for $A\leq CB$ where $C>0$ is a generic constant independent of $h$ or $n$ or $\eps$, and we shall denote $t_s^n$ as some intermediate time value which  may vary line by line in the proof.

\subsection{Main result}

In order to establish the optimal error bounds (with optimal dependence of the $\eps$) of the proposed scheme for solving (\ref{model0}) until a finite time $T>0$ which is independent of $\eps$, we follow the strategy from \cite{Chartier} by introducing the time re-scaling
$t\to t\eps$ which
equivalently formulates (\ref{charged-particle sts}) into a long-time problem
\begin{equation}\label{model}\left\{
\begin{split}
&\dot{x}=\eps v,\quad\dot{v}=v\times B(\eps x)+\eps E(x),\quad 0<t\leq \frac{T}{\eps},\\
&x(0)=x_0,\quad v(0)=v_0.
\end{split}\right.
\end{equation}
Under the assumption that $B(x), E(x)\in C^1(\RR^3)$, %and noting the energy conservation law (\ref{H(x,v)}),
for (\ref{model}) it is clear to have
\begin{equation}\label{regularity}\|x\|_{L^\infty(0,T/\eps)}+\|v\|_{L^\infty(0,T/\eps)}\lesssim 1.
\end{equation}
As another matter of fact, the propagator $\e^{t\widehat{B}(0)}$ generates a periodic flow thanks to the skew-symmetry of $\widehat{B}$, and we shall denote $T_0>0$ as the single  period of it.
The splitting scheme (\ref{TSM0}) under the long-time scaling for solving (\ref{model}) consequently reads
\begin{equation}\label{scheme}
\left\{\begin{split}
x^{n+1}=&x^{n}+
\eps h \e^{h\widehat{B}(\eps x^n)}v^{n}+\frac{h^2\eps^2}{2}\int_{0}^1
E\big(\rho x^{n}+(1-\rho)x^{n+1}\big) d\rho,\quad 0\leq n<\frac{T}{\eps},\\
v^{n+1}=&\e^{h\widehat{B}(\eps x^n)}v^{n}+h\eps\int_{0}^1
E\big(\rho x^{n}+(1-\rho)x^{n+1}\big) d\rho.
\end{split}\right.
\end{equation}
To state the theorem, we introduce the parallel component of the  velocity to the magnetic field
$$
v_\parallel(t) := \frac{B(\eps x(t))}{|B( \eps x(t))|}\, \left(  \frac{B(\eps x(t))}{|B(\eps x(t))|}\cdot v(t) \right),\quad t\geq0, %\qquad v_\perp(t):=v(t)-v_\parallel(t),
$$
 and similarly for the numerical velocity as
$$
v_\parallel^n := \frac{B(\eps x^n)}{|B(\eps x^n)|}\, \left(  \frac{B(\eps x^n)}{|B(\eps x^n)|}\cdot v^n \right), \quad n\geq0.%\qquad v_\perp^n=v^n-v_\parallel^n,
$$

The main convergence result of the splitting scheme is stated as follows.
\begin{theo}\label{Convergence2}(Optimal global convergence) Under the condition that $B(x),E(x)\in C^1(\RR^3)$, let $x^n,\,v^n$ be the numerical solution from the S1-AVF (\ref{scheme}) for solving (\ref{model}) up to $T/\eps$ for some fixed $T>0$, then there exists a constant $N_0>0$ independent of $\eps$, such that when the time step $h=\frac{T_0}{N}$ with some integer $N\geq N_0$, we have the following error bound
\begin{equation}\label{main} \abs{x^{n} -x(t_n)}\lesssim \eps h+N^{-m_0}, \quad
\abs{v_{\parallel}^{n} -v_\parallel(t_n)}\lesssim  \eps h+N^{-m_0},\quad 0\leq n\leq \frac{T}{\eps},
\end{equation}
for some $m_0>0$ arbitrarily large. %Here $v_{\parallel}^n$ and $v_{\parallel}(t_n)$  denote respectively the projection of $v^n$ and $v(t_n)$  in the parallel direction of $B(\eps x(t_n))$.
\end{theo}

The convergence theorems of the other two splitting schemes S1-SV (\ref{TSM1}) and S1-VP (\ref{vps1}) are totally the same as S1-AVF in Theorem \ref{Convergence2} with little modifications in the proof, and so they will be omitted here for simplicity.
Before we step into the proof, we give some important remarks.
\begin{rem}\label{rk: step}
The time step $h=T_0/N$ with some integer $N$ in Theorem \ref{Convergence2} is a technique condition for rigorous proof, which also appeared in \cite{Chartier}. In practice, one only needs $h\lesssim1$ for solving the scaled problem (\ref{model}) to observe the proved optimal error bound as we shall see later in section \ref{sec:num}.
\end{rem}

\begin{rem}
The convergence result of the proposed splitting scheme S1-AVF (\ref{TSM0}) in the original scaling (\ref{model0}) reads equivalently as
$|x^n-x(t_n)|\lesssim h+N^{-m_0}/\eps,\ |v_{\parallel}^{n} -v_\parallel(t_n)|\lesssim h+N^{-m_0}/\eps,$ when $h\lesssim \eps$. Since $N^{-m_0}$ quickly reaches machine accuracy as $N$ increases, so what shows up in practical computing is the uniform part of the error $\bigo(h)$.
\end{rem}

\subsection{Proof of the theorem}

To prove the theorem, we begin by firstly obtaining a coarse estimate for the boundedness of the numerical solution.

\begin{lem}\label{Convergence}Under the condition that $B(x),E(x)\in C^1(\RR^3)$, let $x^n,\,v^n$ be the numerical solution from the S1-AVF (\ref{scheme}) for solving (\ref{model}) up to $T/\eps$ for some fixed $T>0$, then there exists a constant $h_0>0$ independent of $\eps$, such that when the time step $0<h\leq h_0$, we have
$$ \abs{x^{n} -x(t_n)}\lesssim  h, \quad
\abs{v^{n} -v(t_n)}\lesssim  h,\quad 0\leq n\leq T/\eps,$$
and
\begin{equation}\label{bounded}
|x^n|\leq \|x\|_{L^\infty(0,T/\eps)}+1,\quad |v^n|\leq \|v\|_{L^\infty(0,T/\eps)}+1,\quad
0\leq n\leq T/\eps.
\end{equation}
\end{lem}
\begin{proof}

\textbf{Linearized problem.}
First of all, for some $t=t_n+s$ with $n\geq0$, we consider a truncated system of (\ref{model}) as:
\begin{equation}\label{model trun}\left\{
\begin{split}
&\dot{\tilde{x}}^n(s)=\eps \tilde{v}^n(s),\quad 0\leq s\leq h,\\
&\dot{\tilde{v}}^n(s)=\tilde{v}^n(s)\times B(\eps x(t_n))+\eps E(\tilde{x}^n(s)),\\
&\tilde{x}^n(0)=x(t_n),\quad \tilde{v}^n(0)=v(t_n).
\end{split}\right.
\end{equation}
It is also direct to have for all $0\leq n<T/\eps$, there exists a uniform upper bound $C>0$ that depends on
$\|x\|_{L^\infty(0,T/\eps)}$, $\|v\|_{L^\infty(0,T/\eps)}$ and norms of $B$ and $E$ such that
$$\|\tilde{x}^n\|_{L^\infty(0,h)}+\|\tilde{v}^n\|_{L^\infty(0,h)}\leq C.$$
By denoting
$$\zeta_x^n(s):=x(t_n+s)-\tilde{x}^n(s),\quad \zeta_v^n(s):=v(t_n+s)-\tilde{v}^n(s),\quad 0\leq n<T/\eps,$$
and taking the difference between (\ref{model trun}) and (\ref{model}), we get for
$0\leq n<T/\eps$,
\begin{equation}\label{model trun zeta}\left\{
\begin{split}
&\dot{\zeta}_x^n(s)=\eps\zeta_v^n(s),\quad 0\leq s\leq h,\\
&\dot{\zeta}_v^n(s)=\zeta_v^n(s)\times B(\eps x(t_n))+\eps E(x(t_n+s))-\eps
E(\tilde{x}^n(s))+\xi^n_0(s),\\
&\zeta_x^n(0)=\zeta_v^n(0)=0,
\end{split}\right.
\end{equation}
where
\begin{align*}
\xi^n_0(s)=v(t_n+s)\times \left[B(\eps x(t_n+s))-B(\eps x(t_n))\right].
\end{align*}
By Taylor expansion, for some $t_n\leq t_s^n\leq t_n+s$, we have $x(t_n+s)=x(t_n)+s \eps v(t^n_s)$
and then
$$\xi^n_0(s)=s\eps^2 \int_0^1 v(t_n+s)\times \left(\nabla B\left(\eps x(t_n)+  \rho s \eps^2 v(t^n_s)\right)v(t^n_s)\right)d\rho,$$
which clearly indicates that
$$\|\xi^n_0\|_{L^\infty(0,h)}\lesssim \eps^2 h,\quad 0\leq n<T/\eps.$$
By the variation-of-constant formula of (\ref{model trun zeta}), we have
\begin{subequations}
\begin{align}
\zeta_x^n(h)=&\eps\int_0^h \zeta_v^n(s)ds,\quad 0\leq n<\frac{T}{\eps},\\
\zeta_v^n(h)=&\int_0^h
\e^{(h-s)\widehat{B}(\eps x(t_n))}\left[\eps E(x(t_n+s))-\eps E(\tilde{x}^n(s))+\xi^n_0(s)\right]ds \nonumber\\
=&\int_0^h
\e^{(h-s)\widehat{B}(\eps x(t_n))}\left[\eps\int_0^1\nabla E\left(x(t_n+s)+(\rho-1)\zeta^n_x(s)\right)\zeta^n_x(s)d\rho +\xi^n_0(s)\right]ds.\label{voc zeta}
\end{align}
\end{subequations}
The combination of the above two equations gives
\begin{align*}
  \zeta_x^n(h)=&\eps^2\int_0^h \int_0^s
\e^{(s-\sigma)\widehat{B}(\eps x(t_n))}\int_0^1\nabla E\left(x(t_n+\sigma)+(\rho-1)\zeta^n_x(\sigma)\right)\zeta^n_x(\sigma)d\rho\, d\sigma\, ds\\
&+\eps\int_0^h \int_0^s
\e^{(s-\sigma)\widehat{B}(\eps x(t_n))}\xi^n_0(\sigma)d\sigma\, ds,
\end{align*}
which by noting that
$$\left|\eps\int_0^h \int_0^s
\e^{(s-\sigma)\widehat{B}(\eps x(t_n))}\xi^n_0(\sigma)d\sigma ds\right|\lesssim \eps^3h^3,$$
and the standard Bootstrap argument leads to $|\zeta_x^n(s)|\lesssim \eps^3s^3$ for $s\in [0,h]$ with some $h\lesssim1$.
Plugging this estimate into (\ref{voc zeta}) gives for all $0\leq n<T/\eps$,
\begin{equation}\label{zeta est}
  |\zeta_x^n(h)|\lesssim \eps^3h^3,\quad |\zeta_v^n(h)|\lesssim \eps^2 h^2.
\end{equation}
Then to estimate the error of the scheme
$$e_x^{n+1}:=x(t_{n+1})-x^{n+1},\quad e_v^{n+1}:=v(t_{n+1})-v^{n+1},\quad 0\leq n< T/\eps,$$
we shall insert the truncated solution, i.e.
\begin{equation}\label{err eq}
e_x^{n+1}=\tilde{e}_x^n+\zeta_x^n(h),\quad e_v^{n+1}=\tilde{e}_v^n+\zeta_v^n(h),
\end{equation}
and then turn to estimate
$$\tilde{e}^n_x:=\tilde{x}^{n}(h)-x^{n+1},\quad \tilde{e}^n_v:=\tilde{v}^{n}(h)-v^{n+1},\quad 0\leq n< T/\eps.$$

\textbf{Local error.}
Based on the numerical scheme (\ref{scheme}) (or (\ref{TSM0})), we define the local truncation error $\xi_x^n$ and $\xi_v^n$ for $0\leq n<T/\eps$ as
\begin{subequations}\label{local error}
\begin{align}
  \tilde{x}^n(h)=&x(t_n)+h\eps\e^{h\widehat{B}(\eps x(t_n))}v(t_n)+ \frac{h^2\eps^2}{2}\int_0^1E\left(\rho x(t_n)+(1-\rho)\tilde{x}^n(h)\right)d\rho+\xi^n_x,\label{local error x}\\
  \tilde{v}^n(h)=&\e^{h\widehat{B}(\eps x(t_n))}v(t_n)+h\eps \int_0^1E\left(\rho x(t_n)+(1-\rho)\tilde{x}^n(h)\right)d\rho+\xi^n_v.\label{local error v}
  \end{align}
\end{subequations}
By the variation-of-constant formula of the truncated system (\ref{model trun}), we have
\begin{subequations}
\begin{align}
&\tilde{x}^n(h)=x(t_n)+\eps\int_0^h\tilde{v}^n(s)ds,\quad 0\leq n\leq T/\eps,\\
&\tilde{v}^n(h)=\e^{h\widehat{B}(\eps x(t_n))}v(t_n)+\eps \int_0^h
\e^{(h-s)\widehat{B}(\eps x(t_n))}E(\tilde{x}^n(s))ds,\label{vocv}
\end{align}
\end{subequations}
which further implies
\begin{align}\label{vocx}
  \tilde{x}^n(h)=x(t_n)+\eps\int_0^h\e^{s\widehat{B}(\eps x(t_n))}ds\,v(t_n)+\eps^2\int_0^h\int_0^s
\e^{(s-\sigma)\widehat{B}(\eps x(t_n))}E(\tilde{x}^n(\sigma))d\sigma\, ds.
\end{align}

We firstly analyze $\xi_v^n$.
%In (\ref{vocv}), by Taylor expansion, we have
%\begin{equation*}%\label{thm1eq4}
%\frac{1}{\eps}\int_0^h\tilde{B}(x(t_n+s))ds-
%\frac{h}{\eps}\tilde{B}(x(t_n))=\frac{1}{\eps}\int_0^hs\tilde{B}'(x(t^n_s))v(t^n_s)ds,
%\end{equation*}
%where $t^n_s\in[t_n,t_n+s]$ and $\tilde{B}'$ denotes the derivative of $\tilde{B}$. One should note that
%the right hand side of the above equality is a skew-symmetric matrix,
%since both $\tilde{B}(x(t_n))$ and $\int_0^h\tilde{B}(x(t_n+s))ds$ are skew-symmetric. Then we find
%\begin{equation}\label{thm1eq1}
%\e^{\frac{1}{\eps}\int_0^h\tilde{B}(x(t_n+s))ds}-
%\e^{\frac{h}{\eps}\tilde{B}(x(t_n))}=\frac{1}{\eps}
%\int_0^1\e^{\frac{h}{\eps}\tilde{B}(x(t_n))+\frac{\rho}{\eps}\int_0^hs\tilde{B}'(x(t^n_s))v(t^n_s)ds}
%\int_0^hs\tilde{B}'(x(t^n_s))v(t^n_s)ds d\rho.\end{equation}
%Thanks to the skew-symmetry, the matrix $\e^{\frac{h}{\eps}\tilde{B}(x(t_n))+\frac{\rho}{\eps}\int_0^hs\tilde{B}'(x(t^n_s))v(t^n_s)ds}$ is orthogonal.
By Taylor expansion we have in (\ref{vocv})
\begin{align}
&\eps\int_0^h
\e^{(h-s)\widehat{B}(\eps x(t_n))}E(\tilde{x}^n(s))ds\label{thm1eq2}\\
=&\eps\int_0^h\left[I-(s-h)\e^{(h-t^n_s)\widehat{B}(\eps x(t_n))}
\widehat{B}(\eps x(t_n))\right]
E(\tilde{x}^n(s))ds\nonumber\\
=&h\eps\int_0^1E\left(\tilde{x}^n\left((1-\rho)h\right)\right)d\rho
-\eps\int_0^h(s-h)\e^{(h-t^n_s)\widehat{B}(\eps x(t_n))}
\widehat{B}(\eps x(t_n))
E(\tilde{x}^n(s))ds,\nonumber
\end{align}
where $t^n_s\in[s,h]$. Furthermore, by noting that
\begin{align*}
\tilde{x}^n\left((1-\rho)h\right)+\zeta_x^n\left((1-\rho)h\right)
=&x\left(t_n+(1-\rho)h\right)\\
=&x(t_{n+1})-h\rho\eps v(t^n_\rho)\\
=&x(t_{n+1})-\rho\left(x(t_{n+1})-x(t_n)\right)+\rho\left(x(t_{n+1})-x(t_n)\right)-h\rho \eps v(t^n_\rho)\\
=&\rho x(t_n)+(1-\rho)x(t_{n+1})+h\rho\eps\left(v(\tilde{t}^n_\rho)-v(t^n_\rho)\right),
\end{align*}
for some $t^n_\rho,\tilde{t}^n_\rho\in[t_n,t_{n+1}]$, we find
\begin{align}
h\eps\int_0^1E\left(\tilde{x}((1-\rho)h)\right)d\rho
=&h\eps\int_0^1E\left(\rho x(t_n)+(1-\rho)x(t_{n+1})\right)d\rho\label{thm1eq3}\\
&+h\eps\int_0^1\int_0^1E'(s_\sigma)d\sigma \left[h\rho\eps\left(v(\tilde{t}^n_\rho)-v(t^n_\rho)\right)
-\zeta_x^n((1-\rho)h)\right]d\rho,\nonumber
\end{align}
where $E'$ denotes the derivative of $E$ and
\begin{align*}
 s_\sigma=\rho x(t_n)+(1-\rho)x(t_{n+1})+
\sigma h\rho\eps \left(v(\tilde{t}^n_\rho)-v(t^n_\rho)\right)-\sigma \zeta_x^n((1-\rho)h).
\end{align*}
By subtracting (\ref{local error v}) from (\ref{vocv}) and combing (\ref{thm1eq2})-(\ref{thm1eq3}), we find that
\begin{align*}
  \xi_v^n=&-\eps\int_0^h(s-h)\e^{(h-t^n_s)\widehat{B}(\eps x(t_n))}
\widehat{B}(\eps x(t_n))
E(\tilde{x}^n(s))ds\\
&+h\eps\int_0^1\int_0^1E'(s_\sigma)d\sigma \left[h\rho\eps\left(v(\tilde{t}^n_\rho)-v(t^n_\rho)\right)
-\zeta_x^n((1-\rho)h)\right]d\rho,
\end{align*}
which under our assumption clearly implies
\begin{equation}\label{thm1eq8}|\xi_v^n|\lesssim \eps h^2,\quad 0\leq n<T/\eps.
\end{equation}

Next, we estimate $\xi_x^n$. Subtracting (\ref{vocx}) from (\ref{local error x}), we  find
$$\xi^n_x=\xi^n_{x,1}+\xi^n_{x,2},\quad 0\leq n<T/\eps,$$
with
\begin{subequations}
\begin{align}
 \xi^n_{x,1}=&\eps\int_0^h\e^{s\widehat{B}(\eps x(t_n))}ds\,v(t_n)-
 h\eps\e^{h\widehat{B}(\eps x(t_n))}v(t_n),\label{xi_x1}\\
 \xi^n_{x,2}=&\eps^2\int_0^h\int_0^s
\e^{(s-\sigma)\widehat{B}(\eps x(t_n))}E(\tilde{x}^n(\sigma))d\sigma ds
- \frac{h^2\eps^2}{2}\int_0^1E\left(\rho x(t_n)+(1-\rho)\tilde{x}^n(h)\right)d\rho.\label{xi_x2}
\end{align}
\end{subequations}
By the error of the right-rectangle rule, it is direct to see
$$|\xi^n_{x,1}|\lesssim \eps h^2,\quad 0\leq n<T/\eps.$$
For $\xi^n_{x,2}$, firstly we have
\begin{align*}
 &\eps^2\int_0^h\int_0^s
\e^{(s-\sigma)\widehat{B}(\eps x(t_n))}E(\tilde{x}^n(\sigma))d\sigma ds\\
=&\eps^2\int_0^h\int_0^s
\left[I-(\sigma-s)\widehat{B}(\eps x(t_n))\e^{s^n_\sigma\widehat{B}(\eps x(t_n))}\right]E(\tilde{x}^n(\sigma))d\sigma ds,
\end{align*}
for some $s^n_\sigma\in[0,s]$,
and so
\begin{align}
&\eps^2\int_0^h\int_0^s
\e^{(s-\sigma)\widehat{B}(\eps x(t_n))}E\left(\tilde{x}^n(\sigma)\right)d\sigma ds
+\eps^2\int_0^h\int_0^s
(\sigma-s)\widehat{B}(\eps x(t_n))\e^{s^n_\sigma\widehat{B}(\eps x(t_n))}E\left(\tilde{x}^n(\sigma)\right)d\sigma ds\nonumber\\
=&\eps^2\int_0^h\int_0^sE(\tilde{x}^n(\sigma))d\sigma ds=\eps^2\int_0^h s\int_0^1E\left(\tilde{x}^n((1-\rho)s)\right)d\rho ds\nonumber\\
 =&\frac{h^2\eps^2}{2}\int_0^1E\left(\tilde{x}^n((1-\rho)h)\right)d\rho
+\eps^2\int_0^h s(s-h)\int_0^1
E'\left(\tilde{x}^n((1-\rho)\tilde{s}^n_\sigma)\right)
\eps\tilde{v}^n((1-\rho)\tilde{s}^n_\sigma)(1-\rho)d\rho ds, \label{thm1eq4}
\end{align}
for some $\tilde{s}^n_\sigma\in[0,h]$.
Then by plugging (\ref{thm1eq4}) into (\ref{xi_x2}) and further using  (\ref{thm1eq3}), it is clear that
$$|\xi^n_{x,2}|\lesssim \eps^2h^3,\quad 0\leq n<T/\eps,$$
and thus
\begin{equation}\label{thm1eq7}
|\xi_{x}^n|\lesssim\eps h^2,\quad 0\leq n<T/\eps.
\end{equation}

\textbf{Induction for boundedness.} With the above preparation, we now carry out induction proof for the boundedness of the numerical solution (\ref{bounded}).
 For $n=0$, (\ref{bounded}) is obviously true since $x^0=x_0$ and $v^0=v_0$. Then we assume (\ref{bounded}) is true up to some $0\leq m<T/\eps$, and we shall show that  (\ref{bounded}) holds for $m+1$.

For $n\leq m$, subtracting (\ref{local error}) from the scheme (\ref{scheme}), and by further using (\ref{err eq}), we get
\begin{subequations}\label{lm err}
\begin{align}
 &e^{n+1}_x=e^n_x+h\eps\e^{h\widehat{B}(\eps x(t_n))} e^n_v+\eta_x^n+\xi^n_{x}+\zeta_x^n(h),\label{ex eq}\\
 &e^{n+1}_v= \e^{h\widehat{B}(\eps x(t_n))}e^n_v+\eta_v^n+\xi^n_v+\zeta_v^n(h),\quad 0\leq n\leq m,\label{ev eq}
 \end{align}
 \end{subequations}
 where we denote
 \begin{align*}
   \eta_x^n=&h\eps\left(\e^{h\widehat{B}(\eps x(t_n))}-\e^{h\widehat{B}(\eps x^n)}\right)v^n\\
   &+\frac{h^2\eps^2}{2}\int_0^1\left[E\left(\rho x(t_n)+(1-\rho)\tilde{x}^n(h)
   \right)-E\left(\rho x^n+(1-\rho)x^{n+1}
   \right)\right]d\rho,\\
   \eta_v^n=&\left(\e^{h\widehat{B}(\eps x(t_n))}-\e^{h\widehat{B}(\eps x^n)}\right)v^n\\
   &+h\eps\int_0^1\left[E\left(\rho x(t_n)+(1-\rho)\tilde{x}^n(h)
   \right)-E\left(\rho x^n+(1-\rho)x^{n+1}
   \right)\right]d\rho.
 \end{align*}
Thanks to the induction assumption of the boundedness, it is direct to observe that
 \begin{align}\label{eta}
|\eta_x^n|\lesssim h^2\eps^2\left(|e^n_x|+|e^{n+1}_x|+|\zeta_x^n(h)|\right),\quad
|\eta_v^n|\lesssim h\eps\left(|e^n_x|+|e^{n+1}_x|+|\zeta_x^n(h)|\right),\quad 0\leq n<m.
 \end{align}

By taking the absolute value (euclideam norm) on both sides of (\ref{ex eq}) and (\ref{ev eq}) and then using  triangle inequality, noting the orthogonality of the matrix $\e^{h\widehat{B}}$, we get
\begin{align*}
 &|e^{n+1}_x|\leq |e^n_x|+h\eps| e^n_v|+|\eta_x^n|+|\xi^n_{x}|+|\zeta_x^n(h)|,\\
 &|e^{n+1}_v|\leq |e^n_v|+|\eta_v^n|+|\xi^n_v|+|\zeta_v^n(h)|,\quad 0\leq n\leq m.
 \end{align*}
By further adding them together and using (\ref{eta}), we get
 \begin{align*}
   |e^{n+1}_x|+|e^{n+1}_v|-|e^n_x|-|e^n_v|\lesssim &h\eps \left(|e^n_v|+|e^n_x|+|e^{n+1}_x|\right)+|\xi^n_x|
   +|\xi^n_v|+|\zeta^n_x|+|\zeta^n_v|,\quad 0\leq n\leq m.
 \end{align*}
Summing them up for $0\leq n\leq m$ and noting $e^0_x=e^0_v=0$, we obtain
 $$|e^{m+1}_x|+|e^{m+1}_v|\lesssim h\eps \sum_{n=0}^{m}\left(|e^n_v|+|e^n_x|+|e^{n+1}_x|\right)+
 \sum_{n=0}^{m}\left(|\xi^n_x|
   +|\xi^n_v|+|\zeta^n_x|+|\zeta^n_v|\right).
 $$
 By estimates of the truncation errors in (\ref{zeta est}), (\ref{thm1eq8}) and (\ref{thm1eq7}), and noting $mh\eps\lesssim 1$, we get
$$ |e^{m+1}_x|+|e^{m+1}_v|\lesssim h\eps \sum_{n=0}^{m}\left(|e^n_v|+|e^n_x|+|e^{n+1}_x|\right)+h,$$
which then by Gronwall's inequality gives
$$|e^{m+1}_x|+|e^{m+1}_v|\lesssim h,\quad 0\leq m<T/\eps.$$
Since
$$|x^{m+1}|\leq |x(t_{m+1})|+|e^{m+1}_x|,\quad |v^{m+1}|\leq |v(t_{m+1})|+|e^{m+1}_v|,$$
so there exists a generic constant $h_0>0$ independent of $\eps$ and $m$, such that for $0<h\leq h_0$, (\ref{bounded}) holds for $m+1$, which finishes the induction and the proof of this convergence lemma.
\end{proof}

%\begin{lem}\label{Convergence1}(Optimal bound on a period) Under the condition that $B(x),F(x)\in C^1(\RR^3)$, let $x^n,\,v^n$ be the numerical solution from the S1-AVF (\ref{scheme}) for solving (\ref{model}) up to a period $T_0$, then there exists a constant $N_0>0$ independent of $\eps$, such that when the time step $h=\frac{T_0}{N}$ with some integer $N\geq N_0$, we have the following error bound
%$$ \abs{x^{n} -x(t_n)}\lesssim \eps^2 h+\eps N^{-m_0},\quad 0\leq n\leq N,$$
%for some integer $m_0>0$ arbitrarily large.
%\end{lem}

Now, we give the proof of the main convergence result Theorem \ref{Convergence2},  which refines the error bounds to an optimal dependence in $\eps$.

\emph{Proof of Theorem \ref{Convergence2}.}

\begin{proof}
  For any fixed $T>0$, we can have
  $$\frac{T}{\eps}=T_0M+t_r,\quad 0\leq t_r<T_0,$$
 where
  the integer $$M=\left\lfloor\frac{T}{\eps T_0}\right\rfloor=\bigo(1/\eps).$$
  For the integration error on $t_r$, it is just a cumulation of the truncation error (\ref{zeta est}), (\ref{thm1eq8}) and (\ref{thm1eq7}) on a time interval less than one period. So without loss generality, we assume $t_r=0$ in the following proof for simplicity.

  \textbf{Update of notations.}
 First of all, we find the $N_0>0$ by satisfying the condition $h=T_0/N\leq h_0$ given in Lemma \ref{Convergence}, and so when $N\geq N_0$, we have the boundedness (\ref{bounded}). To describe the time scale more clearly, let us renew our notations by denoting $t_n^m$ for $0\leq n\leq N$ as the time grids within the $m$-th period, i.e.
$$t^m_n=mT_0+nh,\quad 0\leq m<M,$$
then we denote the numerical solution from the scheme (\ref{scheme}) at $t_n^m$ as
$$x^{m}_n\approx x(t^m_n),\quad v^{m}_n\approx v(t^m_n),\quad 0\leq m<M,
\quad 0\leq n\leq N,
$$
and the error as
$$e^{n,m}_x=x(t^m_n)-x^{m}_n,\quad e^{n,m}_v=v(t^m_n)-v^{m}_n.$$
Note by our notation, $e_x^{0,m+1}=e_x^{N,m}$ and $e_v^{0,m+1}=e_v^{N,m}$.
Accordingly, the error equation (\ref{lm err}) now reads
\begin{subequations}\label{thm err}
\begin{align}
 &e^{n+1,m}_x=e^{n,m}_x+h\eps\e^{h\widehat{B}(\eps x(t_n^m))} e^{n,m}_v+\eta_x^{n,m}+\xi^{n,m}_{x}+\zeta_x^{n,m}(h),\label{ex eq 2}\\
 &e^{n+1,m}_v= \e^{h\widehat{B}(\eps x(t_n^m))}e^{n,m}_v+\eta_v^{n,m}+\xi^{n,m}_v+\zeta_v^{n,m}(h),\quad 0\leq n\leq N-1,\ 0\leq m<M.\label{ev eq 2}
 \end{align}
 \end{subequations}
The notations for the other error terms are updated in the straightforward manner. For example,
we denote $\xi_{x,1}^{n,m}$ as the local error introduced in (\ref{xi_x1}) at $t_n^m$ level:
\begin{equation}\label{xi10add}
\xi^{n,m}_{x,1}=\eps\int_0^h\e^{s\widehat{B}(\eps x(t_n^m))}ds\,v(t_n^m)-
 h\eps\e^{h\widehat{B}(\eps x(t_n^m))}v(t_n^m).\end{equation}

 Similarly as the proof of Lemma \ref{Convergence}, from the error equation (\ref{thm err}),
 we find
 \begin{align*}
&\frac{1}{\eps}|e^{j,m}_x|-\frac{1}{\eps}|e^{j-1,m}_x|\lesssim  h|e^{j,m}_v|+\eps h\left(
 |e_x^{j,m}|+|e_x^{j-1,m}|\right)+ h^2,\\
& |e^{j,m}_v|-|e^{j-1,m}_v|\lesssim \eps h\left(
 |e_x^{j,m}|+|e_x^{j-1,m}|\right)+\eps h^2,\quad 1\leq j\leq N,\ 0\leq m<M,
 \end{align*}
 where this time we divided (\ref{ex eq 2}) by $\eps$ to gain a better control of error in $v$.
By adding the above two inequalities together and summing up for $j=1,\ldots,n$ for any $1\leq n\leq N$, and then by Gronwall's inequality, we are able to get the estimate of the error within each period:
$$
\frac{1}{\eps}|e^{n,m}_x|+ |e^{n,m}_v|\lesssim h+\frac{1}{\eps}|e^{0,m}_x|+ |e^{0,m}_v|,\quad 1\leq n\leq N,\ 0\leq m<M,
$$
and so by $|e^{0,m}_v|\lesssim h$, we get
\begin{equation}\label{thmeq add}
|e^{n,m}_x|\lesssim \eps h+|e^{0,m}_x|,\quad 1\leq n\leq N,\ 0\leq m<M.
\end{equation}

 \textbf{Refined local error.}
We now refine  the estimate for $\xi_{x,1}^{n,m}$. Directly, we see that
\begin{equation}\label{thm1eq5}
|B(\eps x(t))-B(0)|\lesssim\eps,\quad 0\leq t\leq T/\eps,\end{equation}
and then by comparison with the free flow $\e^{t\widehat{B}(0)}$, it shows
\begin{equation}\label{thm1eq6}
\left|v(mT_0+t)-\e^{t\widehat{B}(0)}v(mT_0)\right|\leq Ct\eps,\quad 0\leq t\leq T_0,\end{equation}
for some constants $C>0$ independent of $\eps$ and $t$.
With these two facts, by denoting $B_0=\widehat{B}(0)$ for short, we split the $\xi_{x,1}^{n,m}$ in (\ref{xi10add}) into two parts:
$$\xi_{x,1}^{n,m}=\xi_{x,1,1}^{n,m}+\xi_{x,1,2}^{n,m},\quad 0\leq n< N,$$ where
\begin{align*}
\xi_{x,1,1}^{n,m}:= \eps\int_0^h\e^{sB_0}ds \e^{t_n^mB_0}v(mT_0)-\eps h\e^{hB_0}\e^{t_n^mB_0}v(mT_0),
\end{align*}
and
\begin{align*}
  \xi_{x,1,2}^{n,m}:=&
  \eps\int_0^h\left(\e^{s\widehat{B}(\eps x(t_n^m))}-\e^{sB_0}\right)ds\,v(t_n^m)
  -
 h\eps\left(\e^{h\widehat{B}(\eps x(t_n^m))}-\e^{hB_0}\right)v(t_n^m)\\
&+\eps\int_0^h\e^{sB_0}ds\left(v(t_n^m)-\e^{t_n^mB_0}v(mT_0)\right)
-h\eps\e^{hB_0}\left(v(t_n^m)
 -\e^{t_n^mB_0}v(mT_0)\right).
\end{align*}

We begin with $\xi_{x,1,2}^{n,m}$.
Clearly by (\ref{thm1eq5}),
$$\left|\e^{s\widehat{B}(\eps x(t_n^m))}-\e^{sB_0}\right|\lesssim s\eps,\quad 0\leq s\leq h.$$
As for the last two terms in $\xi_{x,1,2}^{n,m}$, we first observe that
\begin{align*}
&\eps\int_0^h\e^{sB_0}ds\left(v(t_n^m)-\e^{t_n^mB_0}v(mT_0)\right)
-h\eps\e^{hB_0}\left(v(t_n^m)
 -\e^{t_n^mB_0}v(mT_0)\right)\\
=&\eps\int_0^h(s-h)\e^{t_sB_0}B_0ds\left(v(t_n^m)-\e^{t_n^mB_0}v(mT_0)\right),
\end{align*}
for some $t_s\in [0,h]$. Moreover, thanks to periodicity and (\ref{thm1eq6}), we find
$$v(t_n^m)-\e^{t_n^mB_0}v(mT_0)
=v(mT_0+nh)-\e^{nhB_0}v(mT_0)=O(\eps).$$
Therefore, all together we find
$$\left|\xi_{x,1,2}^{n,m}\right|\lesssim h^2\eps^2.$$
For $\xi_{x,1,1}^{n,m}$, we sum them up for $n=0,\ldots,N-1$, to obtain
\begin{align*}
 \chi^m:=\sum_{n=0}^{N-1}\xi_{x,1,1}^{n,m}
 =\eps \int_0^{T_0}\e^{sB_0}ds v(mT_0)-\eps h\sum_{n=0}^{N-1}\e^{(n+1)h
 B_0}v(mT_0),\quad 0\leq m<M.
\end{align*}
Note $\chi^m$ reads precisely as the quadrature error of trapezoidal rule for the integration of the smooth periodic function $\e^{sB_0}$ on a period, and so
$$|\chi^m|\lesssim \eps N^{-m_0},\quad 0\leq m<M,$$
for some $m_0>0$ arbitrarily large. Thus, in total we find
$$\left|\sum_{n=0}^{N-1}\xi_{x,1}^{n,m}\right|\leq|\chi^m|
+\left|\sum_{n=0}^{N-1}\xi_{x,1,2}^{n,m}\right|\lesssim \eps^2 h+\eps N^{-m_0}.$$

\textbf{Refined error equation.} We now need a clearer description of how the error propagates through each period.
For some $0\leq m<M$, by summing (\ref{ex eq 2}) up for $n=0,\ldots N-1$, we get
 $$e^{N,m}_x=e_x^{0,m}+h\eps\sum_{n=0}^{N-1} \e^{h\widehat{B}(\eps x(t_n^m))}e^{n,m}_v+\sum_{n=0}^{N-1}\left(
 \eta_x^{n,m}+\zeta_x^{n,m}(h)\right)+
 \sum_{n=0}^{N-1}\left(\xi_{x,1}^{n,m}+\xi_{x,2}^{n,m}\right),$$
then by using (\ref{thm1eq5}), we see
\begin{align}\label{thm1eq15}
 e^{N,m}_x=e_x^{0,m}+h\eps\sum_{n=0}^{N-1} \e^{hB_0}e^{n,m}_v+\sum_{n=0}^{N-1}\left(
 \eta_x^{n,m}+\zeta_x^{n,m}(h)\right)+\sum_{n=0}^{N-1}
 \left(\xi_{x,1}^{n,m}+\xi_{x,2}^{n,m}\right)
 +\delta_x^m,
\end{align}
where thanks to $e^{n,m}_v=O(h)$ from Lemma \ref{Convergence},
$$|\delta_x^m|\lesssim \eps^2h^2,\quad 0\leq m< M.$$
On the other hand, similarly by (\ref{thm1eq5}), (\ref{ev eq 2}) can be written as
\begin{align}\label{thm1eq17}
  e^{n,m}_v= \e^{hB_0}e^{n-1,m}_v+\eta_v^{n-1,m}+\xi^{n-1,m}_v+\zeta_v^{n-1,m}(h)+
  \delta_v^{n-1,m},\quad 1\leq n\leq N,\ 0\leq m<M,
\end{align}
where \begin{equation}\label{thm2eq2}
\delta_v^{n-1,m}=\left(\e^{h\widehat{B}(\eps x(t_{n-1}))}-\e^{hB_0}\right)e^{n-1,m}_v,\quad\mbox{and}\quad \left|\delta_v^{n-1,m}\right|\lesssim \eps h^2.\end{equation}
 Recursively from (\ref{thm1eq17}), we find for any $1\leq n\leq N,\ 0\leq m<M$,
\begin{align*}%\label{thm1eq15}
  e^{n,m}_v= \e^{nhB_0}e^{0,m}_v+\sum_{j=0}^{n-1}\e^{(n-1-j)hB_0}
  \left[\eta_v^{j,m}+\xi^{j,m}_v+\zeta_v^{j,m}(h)+
  \delta_v^{j,m}\right],
\end{align*}
and so
\begin{align*}%\label{thm1eq15}
h\eps\sum_{n=0}^{N-1}\e^{hB_0} e^{n,m}_v= h\eps\sum_{n=0}^{N-1}\e^{(n+1)hB_0}e^{0,m}_v+h\eps\sum_{n=0}^{N-1}
\sum_{j=0}^{n-1}\e^{(n-j)hB_0}\left[\eta_v^{j,m}+\xi^{j,m}_v+\zeta_v^{j,m}(h)+
  \delta_v^{j,m}\right].
\end{align*}
Now with the above equation, (\ref{thm1eq15}) can be written as
\begin{align}\label{thm1eq16}
e^{N,m}_x=e_x^{0,m}+h\eps\sum_{n=0}^{N-1} \e^{(n+1)hB_0}e^{0,m}_v+\gamma^m,\quad
0\leq m<M.
\end{align}
where
\begin{align*}
\gamma^m:=&\sum_{n=0}^{N-1}\left(
 \eta_x^{n,m}+\zeta_x^{n,m}(h)\right)+\sum_{n=0}^{N-1}(\xi_{x,1}^{n,m}+\xi_{x,2}^{n,m})
 +\delta_x^m\\
 &+h\eps\sum_{n=0}^{N-1}\sum_{j=0}^{n-1}\e^{(n-j)hB_0}\left[\eta_v^{j,m}+\xi^{j,m}_v+\zeta_v^{j,m}(h)+
  \delta_v^{j,m}\right].
\end{align*}
Noting from (\ref{thm2eq2}), (\ref{eta}), (\ref{thm1eq8}) and (\ref{zeta est}), for the last term in the above we have
$$\left|h\eps\sum_{n=0}^{N-1}\sum_{j=0}^{n-1}\e^{(n-j)hB_0}\left[\eta_v^{j,m}+\xi^{j,m}_v+\zeta_v^{j,m}(h)+
  \delta_v^{j,m}\right]\right|
  \lesssim \eps^2h+\eps^2h\sum_{n=0}^{N-1}\left(|e_x^{n,m}|+|e_x^{n+1,m}|\right),$$
 and therefore we find
  $$|\gamma^m|\lesssim\eps^2h+\eps N^{-m_0}+\eps^2h\sum_{n=0}^{N-1}\left(|e_x^{n,m}|+|e_x^{n+1,m}|\right),\quad 0\leq m<M.$$
By the quadrature error of trapezoidal rule again,  we then deduce from (\ref{thm1eq16})
\begin{align}\label{thm1eq20}
\left|e^{N,m}_x\right|-\left|e_x^{0,m}\right|\lesssim\eps\left|\int_0^{T_0} \e^{sB_0}ds e^{0,m}_v\right|+\eps^2h+\eps N^{-m_0}+\eps^2h\sum_{n=0}^{N-1}\left(|e_x^{n,m}|+|e_x^{n+1,m}|\right),\
0\leq m<M.
\end{align}

By the Rodrigues' formula, we have
\begin{align*}
 \e^{sB_0}e^{0,m}_v=\cos(s|B(0)|)e^{0,m}_v
 +\sin(s|B(0)|)e^{0,m}_v\times \tilde{B}_0+\left(1-\cos(s|B(0)|)\right)\left(\tilde{B}_0\cdot
 e^{0,m}_v\right)\tilde{B}_0,
\end{align*}
where $\tilde{B}_0$ is normalized magnetic field vector at origin, i.e. $\tilde{B}_0=B(0)/|B(0)|$. The integration of the above term over one period only leaves
$$\int_0^{T_0} \e^{sB_0}ds e^{0,m}_v
=T_0 \left(\tilde{B}_0\cdot
 e^{0,m}_v\right)\tilde{B}_0.$$
 Thus, (\ref{thm1eq20}) tells
\begin{align*}
\left|e^{N,m}_x\right|-\left|e_x^{0,m}\right|\lesssim&\eps\left|\left(\tilde{B}_0\cdot
 e^{0,m}_v\right)\tilde{B}_0\right|+\eps^2h+\eps N^{-m_0}+\eps^2h\sum_{n=0}^{N-1}\left(\left|e_x^{n,m}\right|
 +\left|e_x^{n+1,m}\right|\right)\\
\lesssim&\eps\left| e^{0,m}_{v,\parallel}\right|+\eps^2h+\eps N^{-m_0}+\eps^2h\sum_{n=0}^{N-1}\left(\left|e_x^{n,m}\right|
+\left|e_x^{n+1,m}\right|\right),
 \quad
0\leq m<M,
\end{align*}
where  $e^{0,m}_{v,\parallel}$ denotes the error $e^{0,m}_v$ in the parallel direction of the magnetic field $B(\eps x(mT_0))$, i.e.
$$e^{n,m}_{v,\parallel}:=(\tilde{B}^{n,m}\cdot
 e^{n,m}_v)\tilde{B}^{n,m},\quad \tilde{B}^{n,m}:=\frac{B(\eps x(t^m_n))}{\left|B(\eps x(t^m_n))\right|},\quad 0\leq n\leq N,\ 0\leq m<M.$$
Then by (\ref{thmeq add}) and noting $e^{N,m}_x=e^{0,m+1}$, we get
\begin{align}\label{thm1eq18}
\left|e^{0,m+1}_x\right|-\left|e_x^{0,m}\right|
\lesssim\eps\left|e^{0,m}_{v,\parallel}\right|
+\eps^2\left(\left|e_x^{0,m}\right|
+\left|e_x^{0,m+1}\right|\right)
+\eps^2h+\eps N^{-m_0},\quad
0\leq m<M.
\end{align}

Next, we take inner product on both sides of (\ref{ev eq 2}) with the unit vector $\tilde{B}^{n+1,m}$ to get
\begin{equation}\label{thm2eq6}
\left|e^{n+1,m}_{v,\parallel}\right|\leq \left|\tilde{B}^{n+1,m}\cdot\left(\e^{hB(\eps x(t_n^m))}e^{n,m}_v\right)\right|
+\left|\eta_v^{n,m}+\zeta_v^{n,m}(h)\right|
+\left|\xi^{n,m}_v\cdot\tilde{B}^{n+1,m}\right|.
\end{equation}
By noting
$$\tilde{B}^{n+1,m}=\tilde{B}^{n,m}+\bigo(\eps^2h), $$
as well as the Rodrigues' formula, we get
$$\tilde{B}^{n+1,m}\cdot \left(\e^{hB(\eps x(t_n^m))}e^{n,m}_v\right)=e^{n,m}_{v,\parallel}
+\bigo(\eps^2h^2).$$
Then together with (\ref{eta}) and (\ref{zeta est}), we get from (\ref{thm2eq6}) that  for $0\leq n\leq N,\ 0\leq m<M$,
\begin{equation}\label{thm2eq1}
\left|e^{n+1,m}_{v,\parallel}\right|-\left| e^{n,m}_{v,\parallel}\right|\lesssim h\eps
 \left(\left|e_x^{n+1,m}\right|+\left|e_x^{n,m}\right|\right)
+\left|\xi^{n,m}_v\cdot\tilde{B}^{n,m}\right|+\eps^2h^2.\end{equation}
 Recall from (\ref{local error v}) that $\xi^{n,m}_v$ is defined as
 \begin{align*}
\xi^{n,m}_v= \eps \int_0^h
\e^{(h-s)\widehat{B}(\eps x(t_n^m))}E(\tilde{x}^{n,m}(s))ds-
\eps \int_0^h E(\tilde{x}^{n,m}(s))ds,
 \end{align*}
 then the  Rodrigues' formula implies simply
 \begin{align*}
\xi^{n,m}_v\cdot\tilde{B}^{n,m}=0.
 \end{align*}
 %and the  Rodrigues formula implies
% $$\xi^{n,m}_v\cdot\hat{B}_0|\lesssim \eps^2h^2,\quad
% 0\leq n\leq N,\ 0\leq m<M.$$
% For $\delta_v^{n-1,m}$ as defined in (\ref{thm2eq2}), now we estimate it
% as
%  $$|\delta_v^{n-1,m}\cdot\hat{B}_0|
%  \lesssim h\eps |e^{n-1,m}_{v,\parallel}|+\eps^2h^2. $$
 Therefore, (\ref{thm2eq1}) gives
 \begin{equation}\label{thm2eq3}
 \left|e^{n+1,m}_{v,\parallel}\right|-\left| e^{n,m}_{v,\parallel}\right|\lesssim h\eps
 \left(\left|e_x^{n+1,m}\right|+\left|e_x^{n,m}\right|\right)+\eps^2h^2,\quad
 0\leq n<N,\ 0\leq m<M.
  \end{equation}
  %Plugging (\ref{thmeq add}) into the above, we get
%  $$|e^{n,m}_{v,\parallel}|- | e^{n-1,m}_{v,\parallel}|\lesssim h\eps
%|e_x^{0,m}|+\eps^2h^2,\quad
% 1\leq n\leq N,\ 0\leq m<M,$$
% which by Gronwall's inequality indicates
%\begin{equation}\label{thm2eq4}|e^{n,m}_{v,\parallel}|\lesssim |e_x^{0,m}|+|e^{0,m}_{v,\parallel}|+\eps^2 h,\quad 1\leq n\leq N,\ 0\leq m<M.\end{equation}
Summing up (\ref{thm2eq3}) for $n=0,\ldots,N-1$, gives
  $$\left|e^{0,m+1}_{v,\parallel}\right|-\left| e^{0,m}_{v,\parallel}\right|\lesssim h\eps\sum_{n=0}^{N-1}
 \left(\left|e_x^{n+1,m}\right|+\left|e_x^{n,m}\right|\right)+\eps^2h.$$
 Plugging (\ref{thmeq add}) into the above, we get
 \begin{equation}\label{thm2eq5}
 \left|e^{0,m+1}_{v,\parallel}\right|-\left| e^{0,m}_{v,\parallel}\right|\lesssim \eps
 \left(\left|e_x^{0,m}\right|+\left|e_x^{0,m+1}\right|\right)+\eps^2h,\quad 0\leq m<M.
  \end{equation}

  Finally, combining (\ref{thm2eq5}) and (\ref{thm1eq18}), we get
  \begin{align*}%\label{thm1eq18}
&\left|e^{0,m+1}_x\right|+\left|e^{0,m+1}_{v,\parallel}\right|-
\left|e_x^{0,m}\right|
- \left| e^{0,m}_{v,\parallel}\right|\\
\lesssim&\eps\left[\left|e^{0,m}_{v,\parallel}\right|
+\left|e_x^{0,m}\right|
+\left|e_x^{0,m+1}\right|\right]
+\eps^2h+\eps N^{-m_0},\quad
0\leq m<M,
\end{align*}
then by Gronwall' inequality with noting
 $e_x^{0,0}=e_{v,\parallel}^{0,0}=0$, we find
 \begin{align*}%\label{thm1eq18}
\left|e^{0,m}_x\right|+\left|e^{0,m}_{v,\parallel}\right|\lesssim
\eps h+ N^{-m_0},\quad 0\leq m\leq M.
\end{align*}
The estimates at the intermediates time grids, i.e. $e^{n,m}_x$ and
$e^{n,m}_{v,\parallel}$ for $0<n<N$, are direct results of (\ref{thm2eq3}) and (\ref{thmeq add}), and the whole proof is done.

\end{proof}

We finish this section by remarking that
the uniform error bound $\bigo(h)$ appears to be also true for the presented Lie-Trotter type splitting schemes under a general strong magnetic field $B(x)$ in (\ref{charged-particle sts}), based on our numerical evidence. This will be shown in section \ref{sec:num}, but  the above analysis under the general case is more challenging and is still undergoing. As one of the major difficulty, the corresponding $\eta_v^n$ will lose a factor of $\eps$ in (\ref{eta}), which causes stability issue of the error propagation through (\ref{lm err}) up to the $\bigo(1/\eps)$ final time under the approach. This motivates us to consider other approaches for analysis in the next section.

\section{Convergence in general case}\label{sec:analysis2}
%\subsection{The result of convergence}
In the case of general strong magnetic field in the CPD (\ref{charged-particle sts}), we give the following convergence result of the presented splitting schemes.

\begin{theo}\label{them:Convergence} {(Convergence for general strong magnetic field)}
For the general strong  magnetic field $\frac{1}{\eps} B(x)$ with $0<\eps\ll 1$ and under conditions that
\begin{itemize}
  \item[a)] the initial value of  \eqref{charged-particle sts} is assumed to have an $\eps$-independent bound $M$;
  \item[b)] there is  a bounded set $K$ (independent of~$\eps$) such that for $0\le t \le T$ the exact solution $x(t)$ of \eqref{charged-particle sts} stays in $K$;
  \item[c)] the step size $h$ satisfies $h \leq C \eps$ and   the following non-resonance condition is assumed:
\begin{equation}\label{non-res-exact}
\left|\sinc\left( k\frac{h}{2\eps} |B(x(t))|\right)\right| \ge c >0 \qquad\text{for }\ k=1,2;
\end{equation}
\end{itemize}
the global errors of Algorithms \ref{alg:EPS}, \ref{alg:ES} and \ref{alg:VPS} satisfy the bounds
$$ \abs{x^{n} -x(t_n)}\lesssim \eps, \ \  \abs{v_\parallel^{n} -v_\parallel(t_n)}\lesssim \eps,\ \
\abs{v^n - v(t_n)}\lesssim 1,\ \ 0\leq n \leq T/h.$$
The constants before the errors depend on~$M,K,C,c$ and on the bounds of derivatives of~$B$ and~$E$.
\end{theo}
The proof will be given  in the rest part of this section by using   the  technology of modulated Fourier expansion \cite{ICM,Hairer00,Hairer16,hairer2006}. The following key points will be analysed in sequel.
\begin{itemize}
 \item Section \ref{ss:mfe e} presents the modulated Fourier expansion  of the exact solution.

\item Section \ref{ss:mfe n} derives the modulated Fourier expansion of the numerical solution from S1-AVF.

\item Section \ref{ss:proof} proves the result for S1-AVF by comparing the modulated Fourier expansion of the exact solution with that of S1-AVF.

\item Section \ref{ss:proof o} discusses how to modify the proof for S1-AV and S1-VP.
\end{itemize}
Since the modulated Fourier expansion has been used for analysis of charged-particle dynamics in \cite{Hairer16,Hairer2018,lubich19}  , we focus on the novel modifications and the main differences in the proof. %It is noted that in order to obtain the convergence, we consider another phase function for the modulated Fourier expansion which is different from that used in \cite{Hairer2018,lubich19}.

\begin{rem}
We remark that the result of Theorem \ref{them:Convergence} also holds for the maximal ordering scaling case. However, from the proof below, it will be seen that the error bound $\bigo(h)$ of the presented schemes cannot be derived by modulated Fourier expansion unless the restriction of $h$ is strengthened from $h \leq C \eps$ to $h =\bigo (\eps)$.

%{\color{blue}Moreover, the non-resonance condition (\ref{non-res-exact}) is needed in the proof below but it did not appear  in the analysis of section \ref{sec:analysis1}.}

%It should be noted that in the analysis of section \ref{sec:analysis1}, under the condition that $h \leq C \eps$,  the uniform error bound $O(h)$ of the splitting schemes can be proved.
\end{rem}

\subsection{Modulated Fourier
expansion of exact solution}\label{ss:mfe e}
Following \cite{Hairer2018,lubich19}, denote  the eigenvalues and the corresponding normalized  eigenvectors of the linear map $v\mapsto v\times B(x)$ by
$$\lambda_1 = \iu |B(x)|,\quad   \lambda_0 =0,\quad   \lambda_{-1}=-\iu |B(x)|,$$
and  $$v_1(x), \quad v_0(x),\quad v_{-1}(x),$$
respectively. Letting $P_j(x)=v_j(x)v_j(x)^*$ yields the orthogonal projections onto the eigenspaces, which satisfy $P_{-1}(x)+ P_{0}(x)+P_{1}(x)=I$ and
 \begin{align*}
P_{\pm1}(x)\widehat B(x)\alpha=(\pm\iu |B(x)|)P_{\pm1}(x)\alpha,\quad
P_{0}(x)\widehat B(x)\alpha=\mathbf{0},\end{align*}
for any vector $\alpha\in \mathbb{R}^3$.

\begin{lem} \label{thm:mfe-exact}
(See \cite{Hairer2018})
Under the assumptions a) and b) given in Theorem  \ref{them:Convergence}, the exact solution  $x(t)$ of  \eqref{charged-particle sts} can be expressed in the following  modulated Fourier expansion
\begin{equation}\label{mfe-remainder}
x(t) = \!\!\sum_{|k|\le N}\!\! z^k (t)\, \e^{\iu k \phi (t)/\eps}  + R_N (t),
\end{equation}
with an arbitrary truncation index $N\ge 1$ and the phase function $\phi (t)$ which satisfies  {$\dot \phi (t) = | B (z^0(t))|=\bigo(1)$}. Here $z^0 (t)$ describes the motion of the gyrocenter (guiding center) and all  the coefficient functions $z^k (t)$ can be rewritten  in the time-dependent basis $v_j\bigl(z^0(t)\bigr)$:
\begin{equation*}%\label{zetak}
 z^k = z_1^k + z_0^k + z_{-1}^k ,\qquad z_j^k(t) =P_j\bigl(z^0(t)\bigr)z^k(t), \ \ \ \textmd{for}\ \ j=-1,0,1.
\end{equation*}
%Since $x(t)$ is real, we assume $z^{-k} = \overline{z^k}$ for all $k$.
%Together with the fact that $v_{-1}(x) = \overline {v_1(x)}$ and
%$v_0(x)$ is real, it follows
%\begin{equation}\label{zetak-relation}
%z_{-1}^{-k} = \overline{z_1^k},\qquad
%z_0^{-k} = \overline{z_0^k},\qquad
%z_{1}^{-k} = \overline{z_{-1}^k} .
%\end{equation}
This modulated Fourier expansion has the following properties.

(a) The function $z^0 $ satisfies the differential equations
\begin{subequations}
\begin{align}
\ddot z_0^0 &=  P_0(z^0) E(z^0)  +  2\, \dot P_0(z^0) \dot z^0 + \ddot P_0(z^0) z^0 + \bigo (\eps ) , \label{eq-z00}\\
\dot z_{\pm 1}^0 & =   \dot P_{\pm 1}(z^0)  z^0 + \bigo (\eps),\label{eq-zpm0}
\end{align}
\end{subequations}
and $z^{\pm 1}$ are bounded by
\begin{align}\label{eq-zpmpm}
{z_{\pm 1}^{\pm 1} =\bigo(\eps)},\ \ \ z_{j}^{k}=\bigo(\eps^2),\quad \mbox{for}\quad k=\pm1,\ j\neq k.
\end{align}
Moreover, it is true that
\begin{align} \label{eq-z0t} \dot z^0 \times B(z^0) = \bigo(\eps).\end{align}
%%All other coefficient functions $z_j^k$ are given by algebraic expressions depending
%%on $z^0, \dot z_0^0$.

(b) Under the condition that $\phi (0)=0$, the initial values for the differential equations  \eqref{eq-z00}-\eqref{eq-zpm0} are determined  by
\begin{align*}
z^0(0) &= x(0) + \frac{\dot x(0) \times \frac{B(x(0))}{\eps}}{|\frac{B(x(0))}{\eps}|^2} + \bigo(\eps^2)= x(0) +\bigo(\eps),
\\
%z_0^0(0) &=  P_0(x(0)) x(0)+ \bigo (\eps^2 ) , \\
% z_{\pm 1}^0 (0) & =    P_{\pm 1}(x(0)) \Bigl( x(0)
% \pm \,\iu \frac{\eps}{\dot\phi (0)}\dot x(0)\Bigr) + \bigo (\eps^2),\\
 \dot z_0^0(0) &=  P_0(x(0)) \dot x(0)+ \dot P_0(x(0))  x(0)+ \bigo (\eps ) .
\end{align*}

(c) The coefficient function $z^0(t)$ together with its derivatives (up to order $N$) is bounded as $z^0 = \bigo (1)$ and for other $z^k(t)$  together with their derivatives (up to order $N$), they  are bounded as {$z^k = \bigo (\eps^{|k|})\ \textmd{with}\ \abs{k}>1.$} Moreover, these functions are unique up to $\bigo (\eps^{N+1})$.

(d) The bounds of the remainder term $R_N(t)$ and its derivative  are
\begin{equation*}%\label{remainder-est}
R_N(t) = \bigo \left(t^2 \eps^N\right), \quad \dot R_N(t) = \bigo \left(t \eps^N\right),\quad 0\le t\le T.
\end{equation*}

The above constants symbolised by the \mbox{$\bigo$-notation} depend on~$N, T, M$ and on the bounds of derivatives of~$B$ and~$E$, but they are independent of $\eps$ and $t$ with $0\le t \le T$.
\end{lem}

\begin{proof} These results are immediately obtained by considering the section 4 of  \cite{Hairer2018}.

\end{proof}

\subsection{Modulated Fourier
expansion of S1-AVF}\label{ss:mfe n}
In this subsection, we consider  the {\it modulated Fourier expansion} of S1-AVF.
\begin{lem}\label{thm:mfe-num}
Suppose that  the numerical solution $\{ x^n \}$ of the S1-AVF stays in a compact set~$K$ for $0\le nh \le T$. For a fixed, but arbitrary truncation index $N\ge 1$, the non-resonance condition is required
\begin{equation}\label{non-res}
\left|\sinc\left({ \frac{1}{2}} k\eta   |B(x^n)|\right)\right| \ge c >0 \qquad\text{for } k=1,\dots,N+1,
%h | B(x^n)| \le \frac C{N+1} \qquad \hbox{with}\qquad C< 2\pi ,
\end{equation}
where $\eta = h/\eps$ with the bound $C$. Then, $x^n$ admits the following modulated Fourier expansion at $t=nh$
\begin{equation}\label{mfe-formal-num}
x^n = \!\!\sum_{|k|\le N}\!\! \tilde{z}^k (t)\, \e^{\iu k \tilde{\phi} (t)/\eps}  + \tilde{R}_N (t) ,
\end{equation}
where the phase function $\tilde{\phi}$ is given by  \begin{equation}\label{phi num}\dot{\tilde{\phi}} (t) = | B (\tilde{z}^0(t) ) |.\end{equation}

(a) The functions $\tilde{z}_0^0, \tilde{z}_{\pm 1}^0$ satisfy the following differential equations
\begin{subequations}
\begin{align}
  \ddot{\tilde{z}}_0^0 &=  P_0(\tilde{z}^0) E(\tilde{z}^0) +  2\, \dot P_0(\tilde{z}^0)  \dot{ \tilde{z}}^0 + \ddot P_0(\tilde{z}^0) \tilde{z}^0 + \bigo (\eps ) , \label{eq-z00-num}\\
  \dot{\tilde{z}}_{\pm 1}^0 & =   \dot P_{\pm 1}(\tilde{z}^0)  \tilde{z}^0 + \bigo (\eps), \label{eq-zpm0-num}
\end{align}
\end{subequations}
and $\tilde{z}^{\pm 1}$ are bounded by
\begin{align}
 \tilde{z}_{\pm 1}^{\pm 1} =
 \bigo(\eps),\ \ \  \tilde{z}_{j}^{k}=  \bigo(\eps^2),\quad \mbox{for}\quad k=\pm1,\ j\neq k.\label{eq-zpmpm-num}
\end{align}

(b) For the differential equations  \eqref{eq-z00-num}-\eqref{eq-zpm0-num},  their initial values are determined by
\begin{equation}\label{initial-val-all}
\begin{array}{rcl}
\tilde{z}_0^0(0) &=&  P_0(x(0)) x(0)+ {\bigo (\eps)} , \\[1mm]
 \tilde{z}_{\pm 1}^0 (0) & =&    \displaystyle P_{\pm 1}(x(0))   x(0) + \bigo (\eps), \\[1mm]
 \dot{\tilde{z}}_0^0(0) &=&  P_0(x(0)) \dot x(0)+ \dot P_0(x(0))  x(0)+ \bigo (\eps ) .
  \end{array}
\end{equation}

(c)$\,\&\,$(d)  The results  given in (c) and (d) of Lemma~\ref{thm:mfe-exact} are still true for the coefficient functions $\tilde{z}^k(t)$ and for the remainder term $ \tilde{R}_N$, respectively.

The constants symbolised by the \mbox{$\bigo$-notation} are independent of $\eps$ and $n$ with $0\le nh \le T$, but they depend on~$N,C,M,T$ and on bounds of derivatives of~$B$ and~$E$.
\end{lem}
\begin{proof}
(a)   Let $\tilde{x}(t)= \!\!\sum_{|k|\le N} \tilde{z}^k (t)\, \e^{\iu k \tilde{\phi} (t)/\eps}$ and define the operators
\begin{equation*}%\label{new L1L2}
\begin{split}
&\mathcal{L}(hD)=
 \left(\e^{hD}-I\right)- \e^{\eta \widehat B(\tilde{x}(t))}\left(I-\e^{-hD}\right),\quad \mathcal{L}_1(hD,\tau)=
(1-\tau)\e^{hD}+\tau I ,
\end{split}
\end{equation*}where $D$ is the differential operator (see \cite{Hairer16}).
The operator $\mathcal{L}(hD)$ satisfies
\begin{equation}\label{prop L}
\frac{\mathcal{L}(hD)}{h^2} \big( \tilde{z}^k (t)\, \e^{\iu k \tilde{\phi} (t)/\eps}\big)= \e^{\iu k \tilde{\phi} (t)/\eps} \sum_{l\ge 0}
\eps^{l-2} \left[\frac{c_l^k(t)}{\eta}+\frac{d_l^k(t)}{2} +\e^{\eta \widehat B(\tilde{x}(t))}\left(\frac{d_l^k(t)}{2}-\frac{c_l^k(t)}{\eta}\right) \right] \frac{\d^l}{\d t^l} \tilde{z}^k(t) ,
\end{equation}
where some leading coefficients are given by
\begin{equation}\label{coeff-c-d}
\begin{array}{rcl}
c_{2j}^0(t) &=&0, \quad   c_{2j+1}^0(t) = \eta^{2j} /(2j+1)!,\\[2mm]
c_0^k(t) &=& \displaystyle \frac{\iu}{\eta} \sin \left( k\eta \dot \phi (t)\right) -
\eps \frac{k\eta}2  \sin \bigl( k\eta \dot \phi (t)\bigr) \ddot \phi (t)+ \bigo (\eps^2 ), \\ [2mm]
c_1^k(t) &=& \displaystyle \cos \left( k\eta \dot \phi (t)\right) + \bigo (\eps ) ,\\ [2mm]
d_0^0(t)&=&0,\quad d_{2j}^0 (t)= 2\eta^{2j-2}/(2j)!,\quad  d_{2j+1}^0(t) = 0,\\[2mm]
d_0^k(t) &=& \displaystyle -\frac{4}{\eta^2} \sin^2 
\left( \frac{k\eta \dot \phi (t)}2\right) +
\iu\, \eps\, k \cos \left( k\eta \dot \phi (t)\right)  \ddot \phi (t) + \bigo (\eps^2 ), \\[2mm]
d_1^k(t) &=&\displaystyle  \frac {2\,\iu}\eta \sin \left( k\eta \dot \phi (t)\right) + \bigo (\eps ) .
\end{array}
\end{equation}
We insert $\tilde{x}(t)$ into the scheme of S1-AVF and then obtain 
\begin{equation}\label{xxx-num}
\frac{\mathcal{L}(hD)}{h^2} \tilde{x}(t)=
\frac{1}{2}
\int_{0}^{1}E\left(  \mathcal{L}_1(hD,\tau) \tilde{x}(t)\right)d\tau+\frac{1}{2}\e^{\eta \widehat B(\tilde{x}(t))}
\int_{0}^{1}E\left(  \mathcal{L}_1(-hD,\tau) \tilde{x}(t)\right)d\tau.
\end{equation}
Rewriting $\mathcal{L}(hD) \tilde{x}(t)$ and  $\mathcal{L}_1(hD,\tau) \tilde{x}(t)$ in the series of $\e^{\iu k \tilde{\phi} (t)/\eps}$, expanding the nonlinearities around $\tilde{z}^0$, and comparing the coefficients of $\e^{\iu k \tilde{\phi} (t)/\eps}$ yields the construction of the coefficients functions $\tilde{z}^k$. For deriving the first-order convergence, we only need to explicitly present the results of $\tilde{z}^0$ and $\tilde{z}^{\pm1}$.

For $k=0$ we obtain
\begin{equation*}%\label{zpp-num}
\frac{\mathcal{L}(hD)}{h^2} \tilde{z}^0 =
 \displaystyle
\frac{1}{2}
\int_{0}^{1}E\left(  \mathcal{L}_1(hD,\tau) \tilde{z}^0\right)d\tau+\frac{1}{2}\e^{\eta \widehat B(\tilde{z}^0)}
\int_{0}^{1}E\left(\mathcal{L}_1(-hD,\tau) \tilde{z}^0\right)d\tau
+ \bigo (\eps).
\end{equation*}
Then by the property \eqref{prop L} of  $\mathcal{L}(hD)$, it is arrived that
\begin{equation*}
\begin{split}
 P_0(\tilde{z}^0)\frac{\mathcal{L}(hD)}{h^2} \tilde{z}^0 &= P_0(\tilde{z}^0)\ddot{\tilde{z}}^0+ \bigo (h^2)= P_0(\tilde{z}^0)E(\tilde{z}^0)
+ \bigo (\eps),\\
 P_{\pm1}(\tilde{z}^0)\frac{\mathcal{L}(hD)}{h^2} \tilde{z}^0 &=\frac{1}{\varepsilon} \frac{2}{\eta} \sin^2 \left(\frac{\eta}{2}\left| B\left(\tilde{z}^0\right)\right| \right)\left[ 1-\iu\cot\left(\frac{\eta}{2}\left| B \left(\tilde{z}^0\right)\right| \right)\right]P_{\pm1}(\tilde{z}^0)\dot{\tilde{z}}^0+ \bigo (1)\\
 &= P_{\pm1}(\tilde{z}^0)\frac{1}{2}\left(1+e^{\pm \iu  \eta \left| B \left(\tilde{z}^0\right)\right| }\right)E(\tilde{z}^0)+ \bigo (\eps),
\end{split}
\end{equation*}
which gives  \eqref{eq-z00-num} and \eqref{eq-zpm0-num}, respectively.

For $k=\pm1 $ and after multiplication \eqref{xxx-num} with $P_{0 }(\tilde{z}^0)$,  we  look for
the dominant term of  $\frac{\e^{hD}-2I+\e^{-hD}}{h^2}  P_{0}(\tilde{z}^0)\tilde{z}^{\pm 1}\e^{\pm\iu   \tilde{\phi} (t)/\eps},$ which is
$$-4 \sin^2 \left(\frac{\eta}{2}\left| B\left(\tilde{z}^0\right)\right| \right)P_{0}(\tilde{z}^0)\tilde{z}^{\pm 1}\e^{\pm\iu   \tilde{\phi} (t)/\eps}.$$
Multiplying \eqref{xxx-num} with $P_{\pm1 }(\tilde{z}^0)$ and using  \eqref{phi num}, we note that the $\eps^{-2}$ term is annihilated in
$\frac{\left(\e^{hD}-I\right)- \e^{\pm \iu\eta \left| B  (\tilde{z}^0  )\right| }\left(I-\e^{-hD}\right)}{h^2} P_{\pm 1}(\tilde{z}^0)\tilde{z}^{\pm 1}\e^{\pm  \iu   \tilde{\phi} (t)/\eps}.$ The dominant term of this expression becomes the following $\eps^{-1}$ one:
$$ \frac{1}{\varepsilon \eta} \left(e^{\pm\iu  \eta\left| B(\tilde{z}^0)\right| }-1\right)P_{\pm 1}(\tilde{z}^0)\dot{\tilde{z}}^{\pm 1}\e^{\pm  \iu   \tilde{\phi} (t)/\eps}.$$
We extract these two dominant terms from \eqref{xxx-num} and then respectively get the equations of $\tilde{z}_0^{\pm 1}$ and $\dot{\tilde{z}}_{\pm 1}^{\pm 1}$.
These results as well as the initial value of $\tilde{z}_{\pm 1}^{\pm 1}$ given by \eqref{ini zpm1} yield \eqref{eq-zpmpm-num}.

%
%For $k=\pm1 $ and after multiplication \eqref{xxx-num} with $P_{0 }(\tilde{z}^0)$ and $P_{\pm 1 }(\tilde{z}^0)$,  we respectively get
%\begin{equation*}%\label{z1p-num}
%\begin{split}
%& \frac{\e^{hD}-2I+\e^{-hD}}{h^2}  P_{0}\tilde{z}^{\pm 1}\e^{\iu   \tilde{\phi} (t)/\eps}=-4 \sin^2 \left(\frac{\eta}{2}\left| B\left(\tilde{z}^0(t)\right)\right| \right)P_{0}\tilde{z}^{\pm 1}\e^{\iu   \tilde{\phi} (t)/\eps}\\
%  &\left[\e^{\pm\iu  \eta \dot{\tilde{\phi}}}-1-2\left(\e^{\pm\iu \eta \dot{\tilde{\phi}}}+\e^{ 0}\right)^{-1}\left( \e^{\pm\iu \eta \dot{\tilde{\phi}}}-1\right)\right] P_{0}\tilde{z}^{\pm 1}+\bigo (\eps )  \\
% =&2 \iu \sin^2 \left(\frac{\eta}{2}\left| B\left(\tilde{z}^0(t)\right)\right| \right)\left[ \iu  -\tan\left(\pm\frac{\eta}{2}\left| B \left(\tilde{z}^0(t)\right)\right| \right)\right] \tilde{z}_{0}^{\pm 1}+\bigo (\eps )
%= \bigo (h^2 ),\\
%&\left[\e^{\pm\iu  \eta \dot{\tilde{\phi}}}-1-2\left(\e^{\pm\iu \eta \dot{\tilde{\phi}}}+\e^{ \pm \iu\eta \left| B \left(\tilde{z}^0(t) \right)\right| }\right)^{-1}\left( \e^{\pm\iu \eta \dot{\tilde{\phi}}}-1\right)\right] P_{\pm1}\tilde{z}^{\pm 1}  \\
% =&2 \left[\cos\left(\eta \left| B \left(\tilde{z}^0(t)\right)\right| \right)-1\right] \tilde{z}_{\pm1}^{\pm 1}=-4 \sin^2\left(\frac{\eta}{2} \left| B \left(\tilde{z}^0(t)\right)\right| \right)  \tilde{z}_{\pm1}^{\pm 1}+\bigo (\eps )
%=   \bigo (h^2 ) .
%\end{split}\end{equation*}
% This yields \eqref{eq-zpmpm-num}.

(b)  As a consequence of \eqref{mfe-formal-num}, it is obtained that
\begin{equation*}%\label{initial1-num}
x(0) = \tilde{z}^0(0) + \bigl( \tilde{z}^1(0) + \tilde{z}^{-1}(0)\bigr) + \bigo (\eps^2)= \tilde{z}^0(0) + \bigo (\eps).
\end{equation*}
%From $$
%\dot x(t) = v(t)=\frac{2}{h}\left(\e^{hD}+\e^{\eta B(x^n)}\right)^{-1}\left( \e^{hD}-I\right)x(t),$$
%it follows that
%\begin{align*}
%P_0 \bigr( x(0) \bigr)\dot x(0) &=P_0 \left( x(0) \right)\frac{2}{h}\left(\e^{hD}+\e^{\eta B(x^n)}\right)^{-1}\left( \e^{hD}-1\right)x(0)
%\\
%& =P_0 \bigr( x(0) \bigr)\frac{2}{h}\left(\e^{hD}+1\right)^{-1}\left( \e^{hD}-1\right)x(0) +  \bigo(\eps),
%\\
%& = \dot{\tilde{z}}^0_0(0) -P_0 \bigr( x(0) \bigr)\tilde{z}^0(0) + \bigo(\eps).
%\end{align*}
From  
\begin{align} \nonumber
 v^n=&\frac{1}{h}\mathcal{L}_2(hD)x^n+\frac{h}{2}\int_{0}^1
E\left(\rho x^{n-1}+(1-\rho)x^{n}\right) d\rho=\sum_{|k|\le N}\!\! \frac{1}{h}\mathcal{L}_2(hD)\big(\tilde{z}^k (t)\, \e^{\iu k \tilde{\phi} (t)/\eps}\big)+ \bigo(\eps)\\
 = &\!\!\sum_{|k|\le N}\!\!  \e^{\iu k \tilde{\phi} (t)/\eps} \sum_{l\ge 0}
\eps^{l-1} \left(c_l^k(t)-\frac{1}{2}\eta d_l^k(t)\big) \right) \frac{\d^l}{\d t^l} \tilde{z}^k(t)+ \bigo(\eps),\label{vn}
\end{align}
with $\mathcal{L}_2(hD)=1- \e^{-hD}$, it follows that
\begin{align*}
P_0 \bigr( x(0) \bigr)\dot x(0) &=P_0 \left( x(0) \right)\frac{1}{h}\mathcal{L}_2(hD) \tilde{z}^0(0)+ \bigo (\eps)
\\
& = \dot{\tilde{z}}^0_0(0) -\dot P_0(x(0))  \tilde{z}^0(0)+ \bigo(\eps).
\end{align*}
The initial values \eqref{initial-val-all} are determined by these two formulae. The multiplication \eqref{vn} at $t=0$  with $P_{\pm 1} \bigr( x(0) \bigr)$ gives the initial value \begin{align}\label{ini zpm1}
\tilde{z}_{\pm 1}^{\pm 1}(0)=\bigo(\eps).
\end{align}

(c) For $|k|>1$, from \eqref{xxx-num}, the algebraic relations for  $\tilde{z}^{k}$ can be obtained and based on which, the results of part (c)  can be
derived.

(d) For the part (d),  we do not present the details of the proof  since they can be derived by similar arguments as in \cite{Hairer00,Hairer16,Hairer2018,lubich19}.

\end{proof}
\subsection{Proof for S1-AVF}\label{ss:proof}
From the above two lemmas, it is shown that the coefficient functions  of the modulated Fourier expansions of the exact solution and of S1-AVF
satisfy
$$\abs{z^0(t)-\tilde{z}^0(t)}\lesssim \eps,\quad   \abs{\dot{z}^0(t)-\dot{\tilde{z}}^0(t)}\lesssim \eps,\quad \abs{z^{k}(t)-\tilde{z}^{k}(t)}\lesssim  \eps^{|k|},\quad \textmd{with}\quad k\neq0.$$
The phase functions $\phi$ and $\tilde{\phi}$ differ by
$$\abs{\phi(t)-\tilde{\phi}(t)}\lesssim \eps.$$
 These results  lead to $$\abs{x(t_n)-x^n}\lesssim \eps,$$ which shows the $\bigo (\eps)$ error bound for the positions as presented in Theorem~\ref{them:Convergence}.

For  the error bound for the velocities, we need to study the modulated Fourier expansions of the velocity of the exact solution
and of S1-AVF. By Lemma~\ref{thm:mfe-exact}, the velocity of the exact solution is given by
\begin{equation}\label{v-mfe}
v(t)=\dot x(t) = \dot z^0(t) + \frac{\iu\dot\phi(t)}\eps\, \left(z_1^1(t) \, \e^{\iu\phi(t)/\eps} - z_{-1}^{-1}(t) \, \e^{-\iu\phi(t)/\eps}\right) + \bigo(\eps),
\end{equation}
 which implies
$P_{0}(z^0)v(t)=P_{0}(z^0)\dot z^0(t)+\bigo(\eps).$
The modulated Fourier expansion of $v^n$ obtained by S1-AVF satisfies \eqref{vn}.
%\begin{align*}
%v^n &= \frac{1}{h}\mathcal{L}_2(hD)\tilde{x}(t_n)\\
%%&=\frac{2}{h}\mathcal{L}_2(hD)\tilde{z}^0(t_n)+\frac{2}{h}\mathcal{L}_2(hD) \tilde{z}^{\pm 1}(t_n)+ \bigo (\eps^2 ) \\
%&=\frac{1}{h}\mathcal{L}_2(hD)(\tilde{z}^0(t_n)+\tilde{z}^1(t_n)\e^{\iu\tilde{\phi}(t_n)/\eps}+\tilde{z}^{-1}(t_n)\e^{-\iu\tilde{\phi}(t_n)/\eps})+ \bigo (\eps).
%%\\
%%&= \dot{\tilde{z}}^0(t_n)+ \left(I+\e^{\eta \widehat B(x^n)}\right)^{-1}\left(I-\e^{\eta \widehat B(x^n)}\right)\dot{\tilde{z}}^0(t_n)+ \bigo (\eps).
%\end{align*}
%where $\mathcal{L}_2(hD)=1- \e^{-hD}$.
%%Since the part  $(I+\e^{\eta \widehat B(x^n)})^{-1}(I-\e^{\eta \widehat B(x^n)})\dot{\tilde{z}}^0(t_n)$ appears in the result of $v^n$ but not in \eqref{v-mfe}, one only gets
%%$$
%%\abs{v^n - v(t_n)}\lesssim 1.
%%$$
According to
\begin{align*}
 P_{0}(\tilde{z}^0(t_n))v^n &= P_{0}(\tilde{z}^0(t_n)) \dot{\tilde{z}}^0(t_n)+%\tanc\left({ \frac{1}{2}} \eta   | B (\tilde{z}^0(t_n))|\right)\frac{\iu \dot{\tilde{\phi}}(t_n)}\eps\, (\tilde{z}_0^1(t_n)\e^{\iu\tilde{\phi}(t_n)/\eps}-\tilde{z}_0^{-1}(t_n)\e^{-\iu\tilde{\phi}(t_n)/\eps})+
 \bigo (\eps),\\
  P_{\pm1}(\tilde{z}^0(t_n))v^n& = P_{\pm1}(\tilde{z}^0(t_n)) \dot{\tilde{z}}^0(t_n)+ \left[\frac{\iu}{\eta}   \sin\left(\pm\eta \dot{\tilde{\phi}}(t_n)\right) +\frac{2}{\eta} \sin^2\left(\frac{1}{2}\eta \dot{\tilde{\phi}}(t_n)\right)\right] \frac{\tilde{z}_{\pm1}^{\pm1}(t_n)}{\varepsilon}+ \bigo (\eps),
\end{align*}
and the fact that $v_0(x)$ is collinear to $B(x)$, we obtain
$$
\abs{v^n - v(t_n)}\lesssim 1,\quad \mbox{but}\quad \abs{v^n_\parallel - v_\parallel(t_n)}\lesssim \eps.
$$

\subsection{Proof for S1-SV and S1-VP}\label{ss:proof o}
For the methods S1-SV and S1-VP, the equation \eqref{xxx-num} becomes
\begin{equation*}
\frac{\mathcal{L}(hD)}{h^2} \tilde{x}(t)= \displaystyle
\frac{1}{2}  \left(\e^{\frac{h}{\eps}\widehat B (\tilde{x}(t))}+I\right)E(\tilde{x}(t)),
\end{equation*}
and
\begin{equation*}
\frac{\mathcal{L}(hD)}{h^2} \tilde{x}(t)= \displaystyle
 \varphi_1\left(\frac{h}{\eps}\widehat B (\tilde{x}(t))\right)  E(\tilde{x}(t)),
\end{equation*}
respectively. By using this result as well as the relationship between $x^n$ and $v^n$ determined by each scheme, and by some adaptations of the proofs of the above two subsections, the first-order convergence in $x$ and $v_\parallel$ of  S1-SV and S1-VP remains true. Here we omit the details for brevity.

\section{Numerical result} \label{sec:num}
In this section, we present numerical results of the presented Lie-Trotter type schemes. We first conduct numerical experiments to show the accuracy of the schemes under different $\eps\in(0,1)$ and then we address their efficiency and conservation property.

To test the convergence result of the splitting schemes, we solve the CPD till $T=t_n=1$ numerically and compute the relative error:
\begin{equation}\label{err}
  error:=\frac{\abs{x^n-x(t_n)}}{\abs{x(t_n)}}
  +\frac{\abs{v^n_\parallel-v_\parallel(t_n)}}{\abs{v_\parallel(t_n)}}.
\end{equation}
The reference solution is obtained by using ``ode45" of MATLAB. For the  implicit scheme S1-AVF, we apply  the two-point Gauss-Legendre's
rule to the integral in \eqref{TSM0} and use standard fixed point iteration as nonlinear solver in the practical computations. We
set $10^{-16}$ as the error tolerance and $1000$ as the maximum number of each iteration.

\vskip2mm\noindent\textbf{Problem 1. (Maximal ordering scaling)} The first illustrative numerical experiment is devoted to the charged-particle motion
in a magnetic field with the maximal ordering scaling %(\cite{lubich19})
$$\frac{1}{\eps}B(\eps x) =\frac{1}{\eps}\begin{pmatrix}
                       \cos(\eps x_2)\\ 1+\sin(\eps x_3)\\ \cos(\eps x_1)
                      \end{pmatrix}+\begin{pmatrix}
                       -x_1\\ 0\\ x_3
                      \end{pmatrix},$$ and the electric field $E(x)=-\nabla_x U(x)$ with the potential
$U(x)=\frac{1}{\sqrt{x_1^2+x_2^2}}.$ We choose the initial values  as   $x(0)=(\frac 1 3,\frac 1 4,\frac 1
2)^{\intercal}$ and $v(0)=(\frac 2 5,\frac 2 3,1)^{\intercal}$.  The errors (\ref{err}) of the three Lie-Trotter type splitting schemes, i.e. S1-AVF (\ref{TSM0}), S1-SV (\ref{TSM1}) and S1-VP (\ref{vps1}) at $T=1$ are shown in Figure \ref{fig11}.

Clearly from the numerical results in Figure \ref{fig11}, we can see that

1) The three splitting schemes all show the uniform first order accuracy for the varying $\eps\in(0,1)$ in the position $x$ and $v_\parallel$. This verifies the theoretical result in Theorem \ref{Convergence2} and indicates that the error estimate is optimal. In addition, the choice of  the step size in this problem which is not the integer partition of the period illustrates Remark \ref{rk: step}.

2) The proposed S1-AVF or S1-SV are more accurate than S1-VP, and the errors of S1-AVF and S1-SV are very close.

 \begin{figure}[t!]
$$\begin{array}{cc}
\psfig{figure=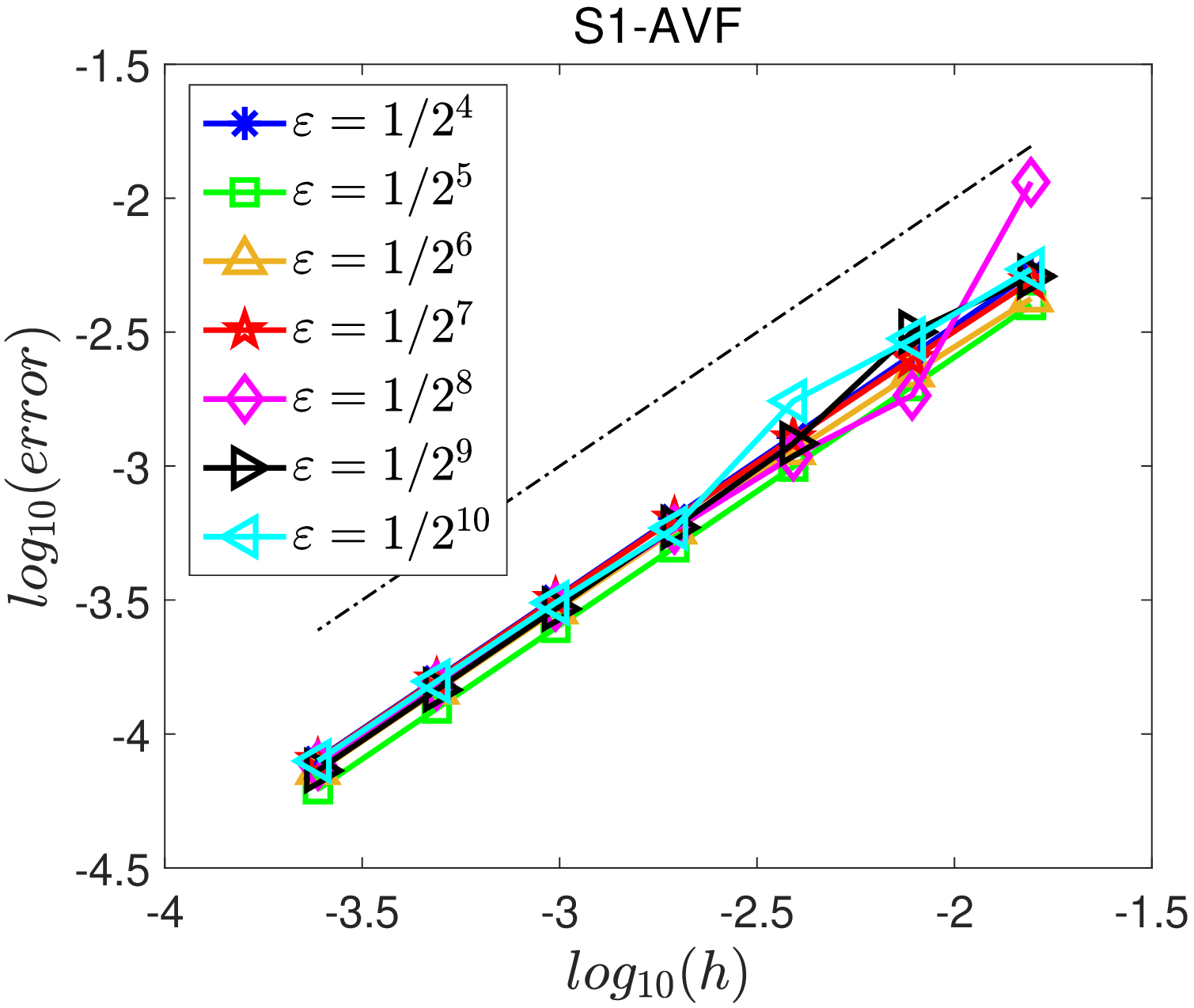,height=5.0cm,width=5cm}
\psfig{figure=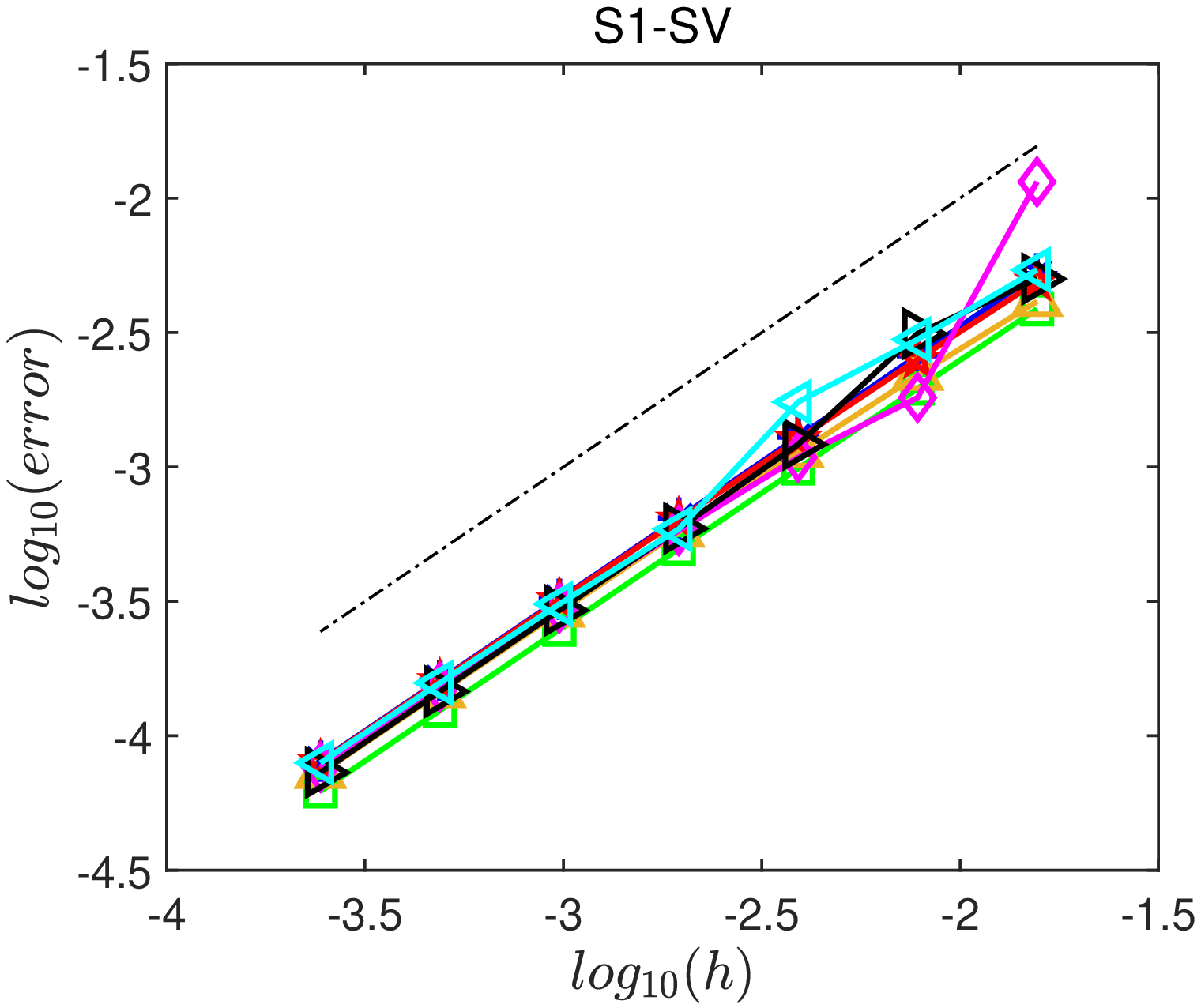,height=5cm,width=5cm}
\psfig{figure=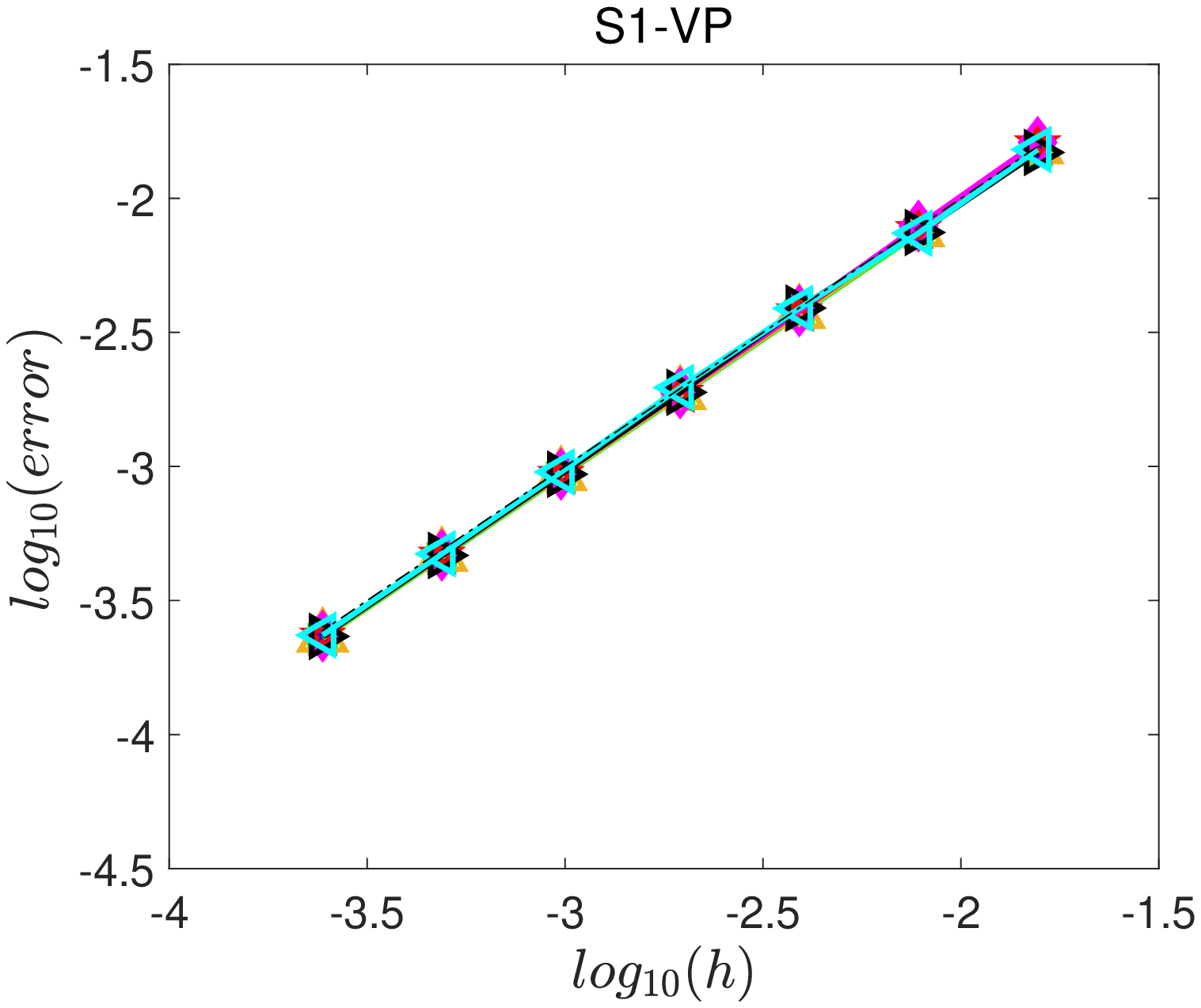,height=5cm,width=5cm}
\end{array}$$
\caption{ The error (\ref{err}) of the three splitting schemes in Problem 1 with step size $h=1/2^k$ for $k=6,\ldots,12$ under different $\eps$ (the dash-dot line is slope one).}\label{fig11}
\end{figure}
%\begin{figure}[ptb]
%\centering\tabcolsep=2mm
%\begin{tabular}
%[l]{lll}%
%\includegraphics[width=5.0cm,height=5.0cm]{err11}
%\includegraphics[width=5.0cm,height=5.0cm]{err12}
%\includegraphics[width=5.0cm,height=5.0cm]{err13}
%\end{tabular}
%\caption{Problem 1 (color figures online). The logarithm of the errors against the logarithm of $h=1/2^k$ for $k=6,\ldots,12$. The slope of the  dotted line is one.}%
%\label{fig11}%
%\end{figure}

%\begin{figure}[ptb]
%\centering\tabcolsep=2mm
%\begin{tabular}
%[l]{lll}%
%\includegraphics[width=5.0cm,height=5.0cm]{h11}
%\includegraphics[width=5.0cm,height=5.0cm]{h12}
%\includegraphics[width=5.0cm,height=5.0cm]{h13}
%\end{tabular}
%\caption{Problem 1 (color figures online). The logarithm of the  energy errors $HE=\frac{\abs{H(x^n,v^n)-H(x^0,v^0)}}{\abs{H(x^0,v^0)}}$ against $t$ with $h=0.01$.}%
%\label{fig12}%
%\end{figure}
%\begin{figure}[t!]
%$$\begin{array}{cc}
%\psfig{figure=h11.eps,height=5.0cm,width=5cm}
%\psfig{figure=h12.eps,height=5cm,width=5cm}
%\psfig{figure=h13.eps,height=5cm,width=5cm}
%\end{array}$$
%\caption{Problem 1: logarithm of the energy errors $HE=\frac{\abs{H(x^n,v^n)-H(x^0,v^0)}}{\abs{H(x^0,v^0)}}$ against $t$ with $h=0.01$.}\label{fig12}
%\end{figure}
\begin{figure}[t!]
$$\begin{array}{cc}
\psfig{figure=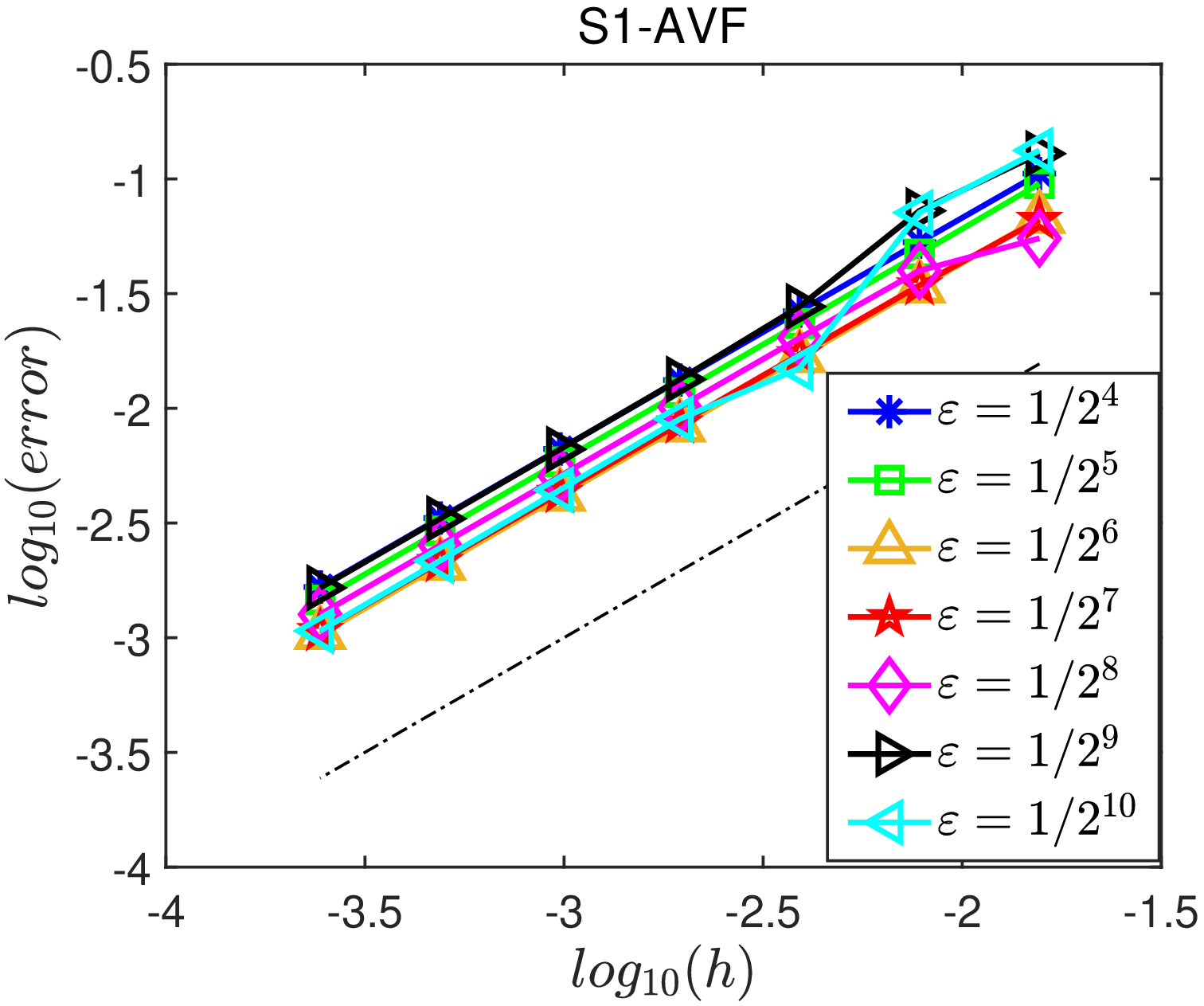,height=5.0cm,width=5cm}
\psfig{figure=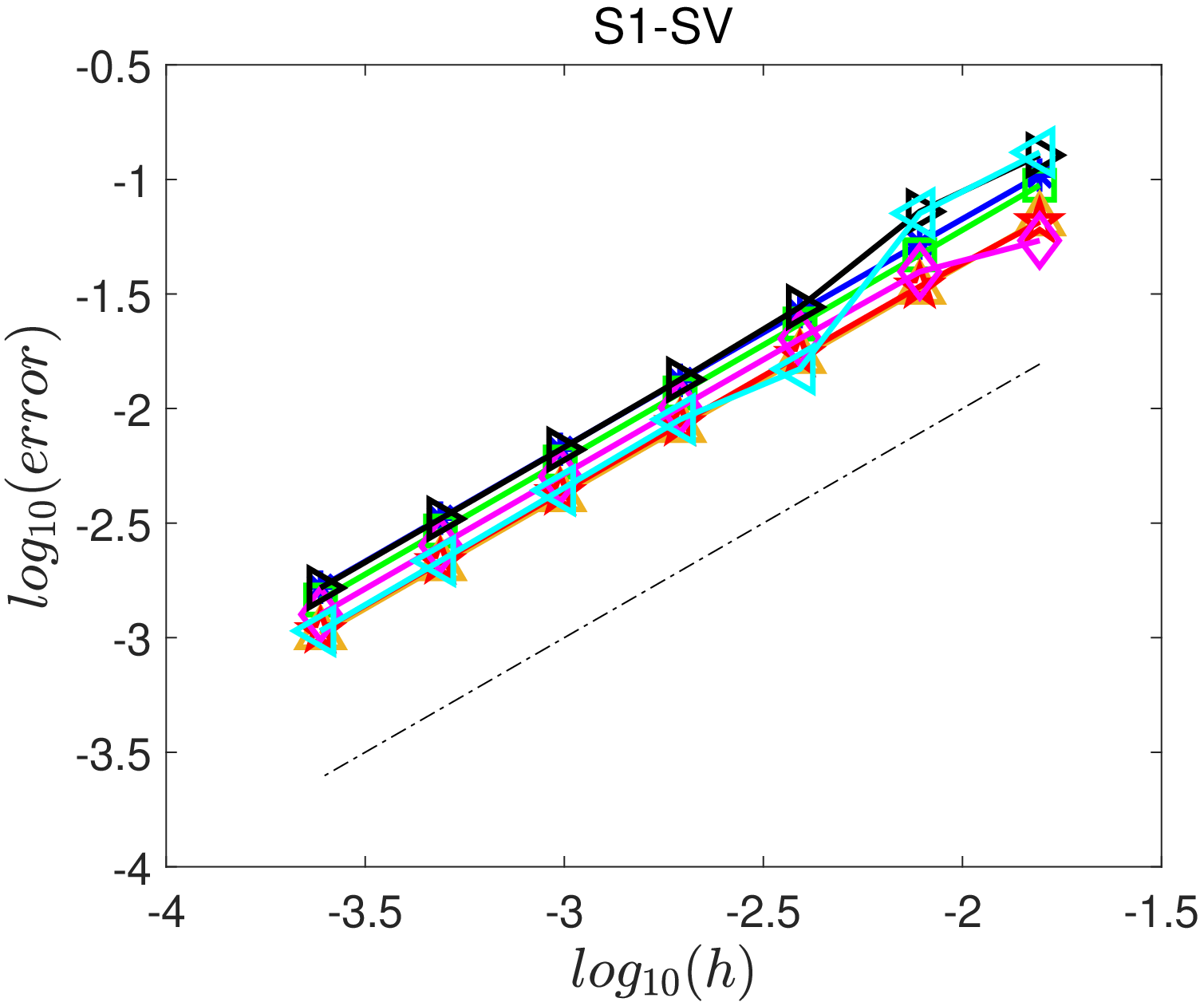,height=5cm,width=5cm}
\psfig{figure=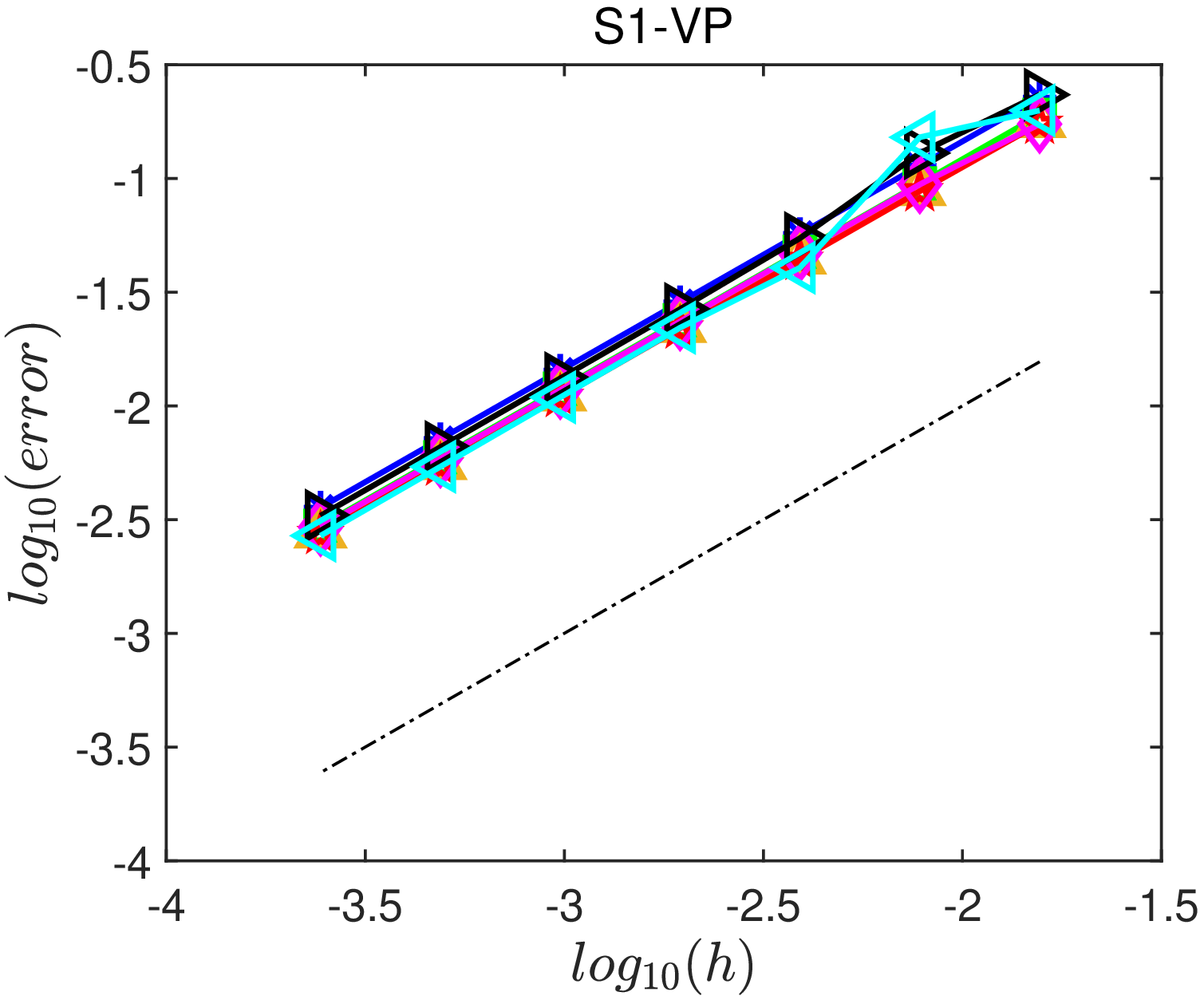,height=5cm,width=5cm}
\end{array}$$
\caption{The error (\ref{err}) of the three splitting schemes in Problem 2 with step size $h=1/2^k$ for $k=6,\ldots,12$ under different $\eps$ (the dash-dot line is slope one). }\label{fig21}
\end{figure}

\begin{figure}[t!]
$$\begin{array}{cc}
\psfig{figure=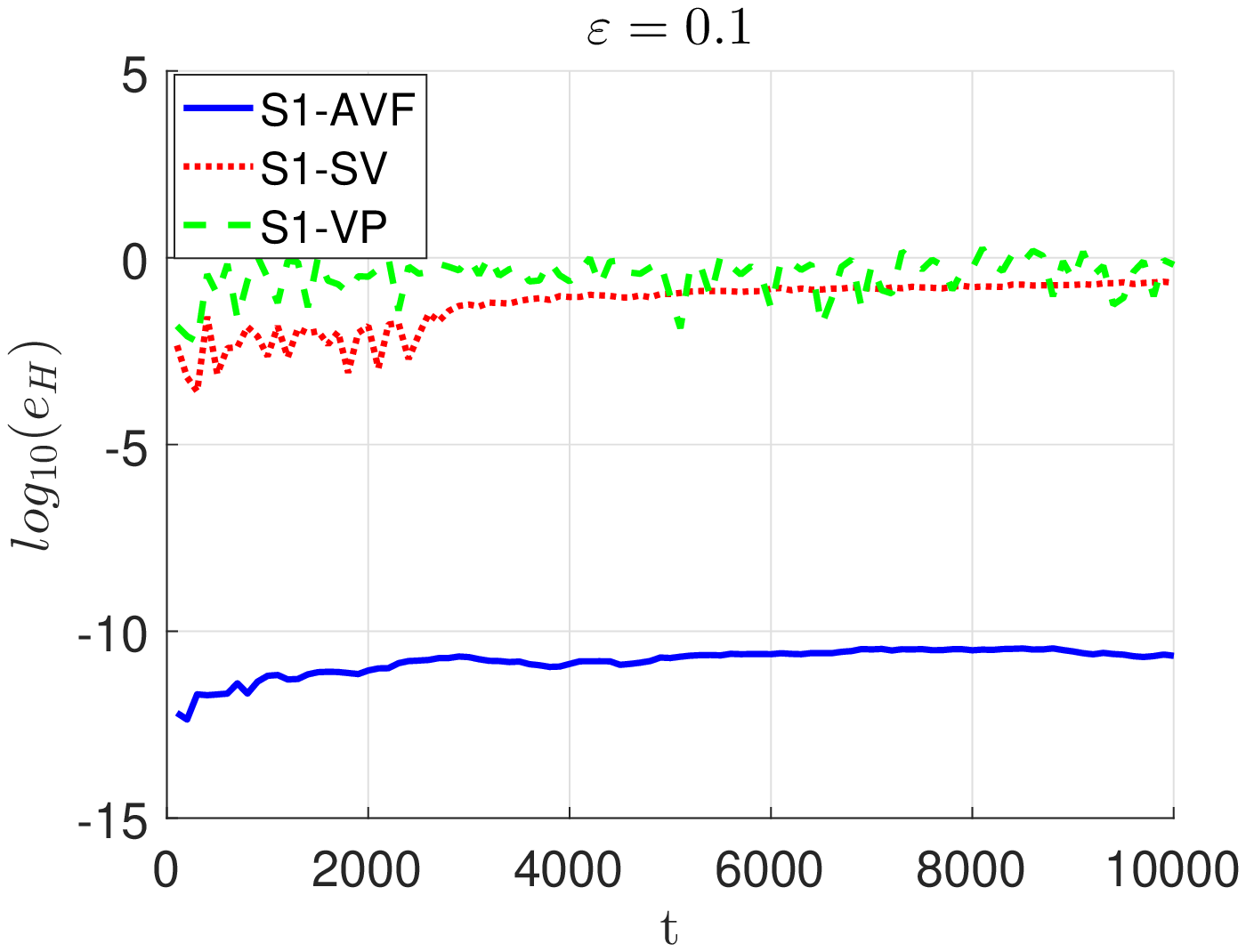,height=5.0cm,width=5cm}
\psfig{figure=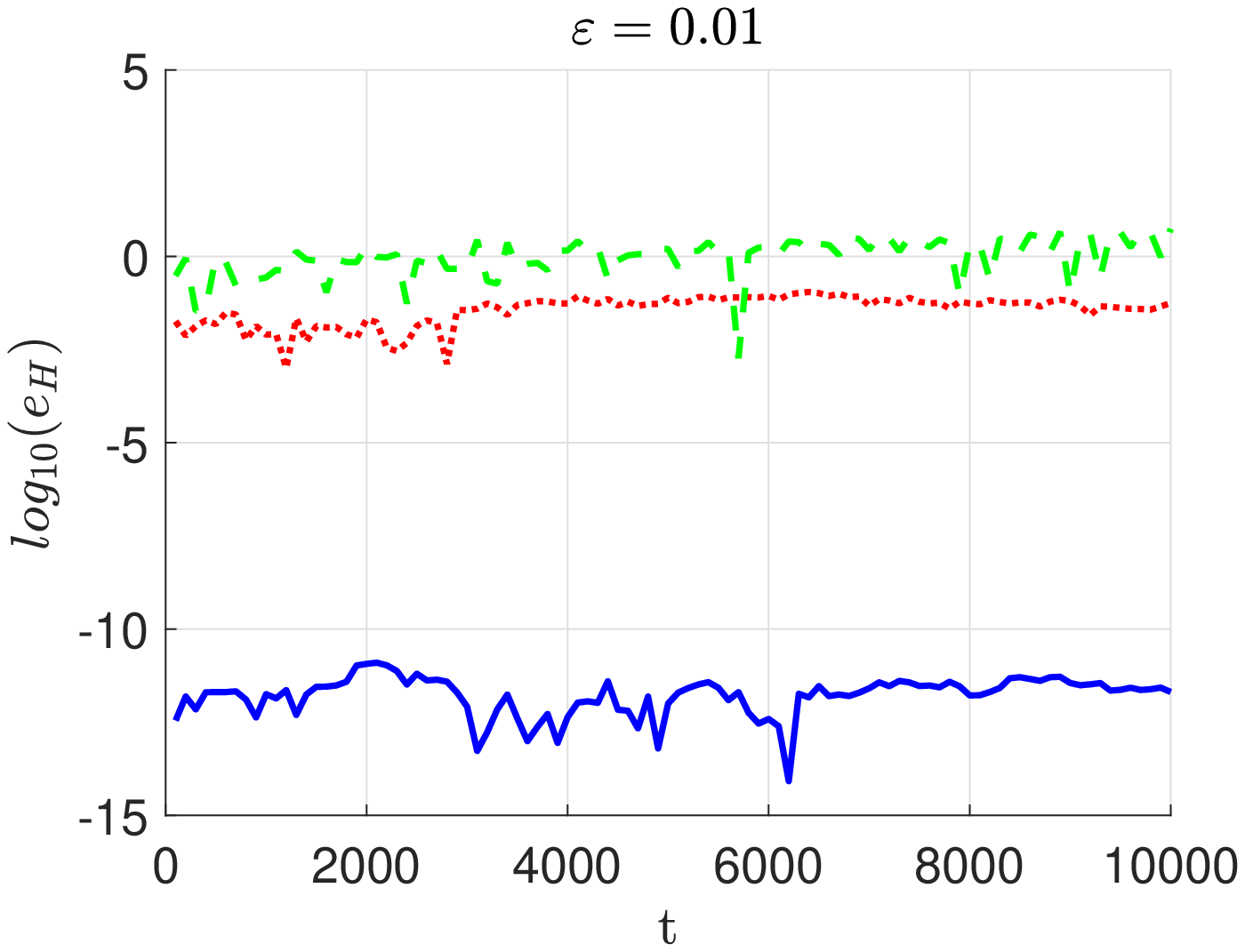,height=5cm,width=5cm}
\psfig{figure=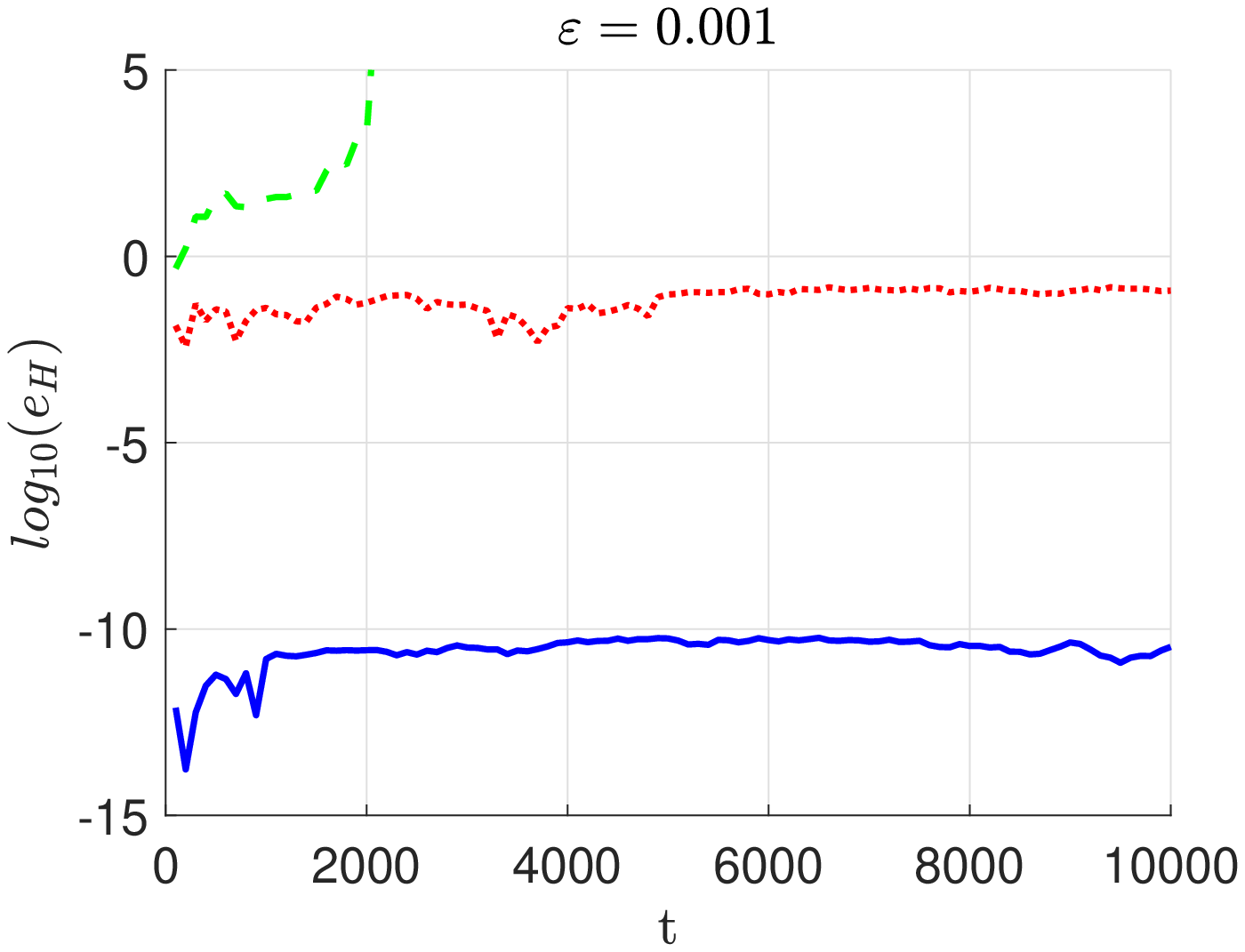,height=5cm,width=5cm}
\end{array}$$
\caption{Evolution of the energy error (\ref{eH}) as function of time $t$ under $h=0.01$ in Problem 2. }\label{fig22}
\end{figure}
%\begin{figure}[ptb]
%\centering\tabcolsep=2mm
%\begin{tabular}
%[l]{lll}%
%\includegraphics[width=5.0cm,height=5.0cm]{h21}
%\includegraphics[width=5.0cm,height=5.0cm]{h22}
%\includegraphics[width=5.0cm,height=5.0cm]{h23}
%\end{tabular}
%\caption{Problem 2 (color figures online). The logarithm of the  energy errors $HE=\frac{\abs{H(x^n,v^n)-H(x^0,v^0)}}{\abs{H(x^0,v^0)}}$ against $t$ with $h=0.01$.}%
%\label{fig22}%
%\end{figure}
\begin{figure}[h!]
$$\begin{array}{cc}
\psfig{figure=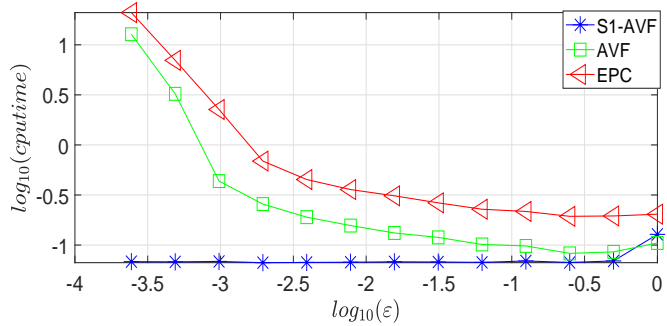,height=4.5cm,width=10cm}
\end{array}$$
\caption{Cputime of the energy-preserving schemes under the same number of iterations with $\eps=1/2^k$ for $k=0,\ldots,12$. }\label{fig23}
\end{figure}

\vskip2mm\noindent\textbf{Problem 2. (General strong magnetic field)} In the second numerical experiment, we consider  the charged-particle motion in
the general magnetic field  \cite{Hairer2018}
$$\frac{1}{\eps}B(x)=\nabla \times \frac{1}{4\eps}\big(x_{3}^2-x_{2}^2, x_{3}^2-x_{1}^2, x_{2}^2-x_{1}^2\big)^{\intercal}=\frac{1}{2\eps}(x_{2}-x_{3},x_{1}+x_{3},x_{2}-x_{1})^{\intercal},
 $$
and the electric field $E(x)=-\nabla_x U(x)$ with the potential
$U(x)=x_{1}^{3}-x_{2}^{3}+x_{1}^{4}/5 +x_{2}^{4}+x_{3}^{4}.$
 The initial values are given by  $x(0)=(0.6,1,-1)^{\intercal}$ and
  $v(0)=(-1,0.5,0.6)^{\intercal}$. Figures \ref{fig21} and \ref{fig22} respectively present the errors (\ref{err}) in the solution at $T=1$ and the errors in the energy
  \begin{equation}\label{eH}
e_H:=\frac{\abs{H(x^n,v^n)-H(x^0,v^0)}}{\abs{H(x^0,v^0)}},
\end{equation}
on a long time interval.

  To illustrate the efficiency of  the proposed S1-AVF, we choose two popular energy-preserving methods from the literature: the direct AVF method \eqref{AVF} for CPD (\ref{charged-particle sts}) and the fourth order energy-preserving collocation method (denoted by EPC) \cite{Hairer000} for comparison. We use the same standard fixed point iteration for all the schemes and set $10^{-16}$ as the error tolerance. %and $1000$ as the maximum number of iterations.
   The system (\ref{charged-particle sts}) is integrated till $T=1$ by each of the method under the same step size $h=1/2^{10}$, and the computational time (cputime) is displayed in Figure \ref{fig23}. This test is conducted in a sequential program in MATLAB on a laptop
   ThinkPad X1 Carbon (CPU: Intel (R) Core (TM) i5-5200U CPU @ 2.20 GHz, Memory: 8 GB, Os: Microsoft Windows 7 with 64bit).

Based on the numerical results in Figures \ref{fig22}\&\ref{fig23}, we can draw the following observations.

1) Under the general strong magnetic field, the three presented splitting schemes  (\ref{TSM0}), (\ref{TSM1}) and (\ref{vps1}) still show the uniform first order error bound $\bigo(h)$ in $x$ and $v_{\parallel}$. Their performances are very similarly as in the maximal order case. This would require a more delicate analysis which is going to be our future work.

2) S1-AVF preserves the energy (\ref{H(x,v)}) to machine accuracy over long times. Between the other two methods, S1-SV has smaller energy error and better long-time behaviour than S1-VP. In comparison with other classical energy preserving methods, the computational cost of S1-AVF is uniform for $\eps\in (0,1]$. Hence, it is more efficient for CPD (\ref{charged-particle sts}) in the strong magnetic field regime.

%\begin{figure}[ptb]
%\centering\tabcolsep=2mm
%\begin{tabular}
%[l]{lll}%
%\includegraphics[width=5.0cm,height=5.0cm]{err21}
%\includegraphics[width=5.0cm,height=5.0cm]{err22}
%\includegraphics[width=5.0cm,height=5.0cm]{err23}
%\end{tabular}
%\caption{Problem 2 (color figures online). The logarithm of the  errors  against the logarithm of $h=1/2^k$ for $k=6,\ldots,12$. The slope of the  dotted line is one.}%
%\label{fig21}%
%\end{figure}

%\begin{figure}[ptb]
%\centering\tabcolsep=2mm
%\begin{tabular}
%[l]{lll}%
%\includegraphics[width=5.0cm,height=5.0cm]{cpu1}
%\end{tabular}
%\caption{Problem 2 (color figures online). The logarithm of the CPU time against the logarithm of $\epsilon=1/2^k$ for $k=0,\ldots,12$..}%
%\label{fig23}%
%\end{figure}

\section{Conclusion}\label{sec:con}
In this paper, we considered the numerical solution of the charged-particle dynamics that involve a small parameter $\eps\in(0,1]$ inversely proportional to the  strength of the external magnetic field. Firstly, a novel splitting scheme that preserves the exact energy of the system was proposed, and its computational cost per step is uniform in $\eps\in(0,1]$. Then under the maximal ordering scaling, by using averaging technique, we established a uniform and optimal first order error bound for the proposed method in the position variable and the parallel part of the velocity variable to the magnetic field. For the general strong magnetic field case, we applied the modulated Fourier expansion for the error analysis of the proposed scheme, and a convergence result in $\eps$ was obtained. Our results  in fact are true for a class of Lie-Trotter type splitting schemes. Numerical experiments were conducted to illustrate the accuracy and efficiency of the scheme.
\section*{Acknowledgements}
 We would like to thank Christian Lubich for valuable
comments and suggestions on the work. X. Zhao is partially supported by the Natural Science Foundation of Hubei Province No. 2019CFA007 and the NSFC 11901440.


\begin{thebibliography}{99}
%%%%%%%%%%%%%%%%%%%%%%%%%%%%%%%%%%%%%%%%%
\bibitem {Arnold97} {\sc V.I. Arnold, V.V. Kozlov, A.I. Neishtadt}, Mathematical
Aspects of Classical and Celestial Mechanics, Springer, Berlin, 1997.

 \bibitem {Benettin94}{\sc G. Benettin, P. Sempio}, Adiabatic invariants and trapping
of a point charge in a strong nonuniform magnetic field,
Nonlinearity 7 (1994), pp. 281-304.

\bibitem {Boris1970} {\sc J.P. Boris}, Relativistic plasma simulation-optimization of a hybrid
code, Proceeding of Fourth Conference on Numerical Simulations of
Plasmas (1970), pp. 3-67.

 \bibitem{L. Brugnano2019} {\sc L. Brugnano, J.I. Montijano, L. R\'{a}ndz}, High-order energy-conserving line integral methods for charged particle dynamics, J. Comput. Phys. 396 (2019), pp. 209-227.

\bibitem{scaling1}
{\sc A.J. Brizard, T.S. Hahm}, Foundations of nonlinear gyrokinetic Theory, Rev. Modern Phys. 79
(2007), pp. 421-468.

\bibitem {Cary2009} {\sc J.R. Cary, A.J. Brizard}, Hamiltonian theory of guiding-center
motion, Rev. Modern Phys. 81  (2009), pp. 693-738.

\bibitem{Chartier} {\sc Ph. Chartier, F. M\'{e}hats, M. Thalhammer, Y. Zhang},
Improved error estimates for splitting methods applied to highly-oscillatory nonlinear Schr\"{o}dinger equations, Math. Comp. 85 (2016), pp. 2863-2885.

\bibitem{VP1}{\sc Ph. Chartier, N. Crouseilles, M. Lemou, F. M\'ehats, X. Zhao}, Uniformly accurate methods for Vlasov equations with non-homogeneous strong magnetic field, Math. Comp. 88 (2019), pp. 2697-2736.

\bibitem{VP2} {\sc Ph. Chartier, N. Crouseilles, X. Zhao}, Numerical methods for the two-dimensional Vlasov-Poisson equation in the finite Larmor radius approximation regime, J. Comput. Phys. 375 (2018),  pp. 619-640.

\bibitem{CPC} {\sc N. Crouseilles, S.A. Hirstoaga, X. Zhao}, Multiscale Particle-In-Cell methods
and comparisons for the long-time two-dimensional Vlasov-Poisson equation with
strong magnetic field, Comput. Phys. Comm. 222 (2018), pp. 136–151.

\bibitem{Zhao}
{\sc Ph. Chartier, N. Crouseilles, M. Lemou, F. M\'ehats, X. Zhao},
Uniformly accurate methods for three dimensional Vlasov equations under strong magnetic field with varying direction, SIAM J. Sci. Compt.  42 (2020), pp. B520-B547.

\bibitem{VP3} {\sc N. Crouseilles, M. Lemou, F. M\'ehats, X. Zhao}, Uniformly accurate Particle-in-Cell method for the long time two-dimensional Vlasov-Poisson equation with uniform strong magnetic field, J. Comput. Phys. 346 (2017), pp. 172-190.

\bibitem{VP4} {\sc F. Filbet, M. Rodrigues}, Asymptotically stable particle-in-cell methods for the Vlasov-Poisson system with a strong external magnetic field, SIAM J. Numer. Anal. 54 (2016), pp. 1120-1146.

\bibitem{VP5}   {\sc F. Filbet, M. Rodrigues}, Asymptotically preserving particle-in-cell methods for inhomogeneous strongly magnetized plasmas, SIAM J. Numer. Anal. 55 (2017), pp. 2416-2443.

\bibitem{VP9}{\sc F. Filbet, M. Rodrigues, H. Zakerzadeh},  Convergence analysis of asymptotic preserving schemes for strongly magnetized plasmas, arXiv:2003.08104v1 [math.NA].
    
\bibitem{VP-filbet} {\sc F. Filbet, T. Xiong, E. Sonnendr\"{u}cker}, On the Vlasov-Maxwell system with a strong magnetic field, SIAM J. Applied Mathematics 78 (2018), pp. 1030-1055.

 \bibitem{VP6}   {\sc  E. Fr\'enod, F. Salvarani and E. Sonnendr\"ucker}, Long time simulation of a beam in a periodic focusing channel via a two-scale PIC-method, Math. Models Methods Appl. Sci. 19 (2009), pp. 175-197.

 \bibitem{VP8}
{\sc E. Fr\'{e}nod, S. Hirstoaga, M. Lutz, E. Sonnendr\"{u}cker}, Long time behavior of an exponential integrator for a Vlasov-Poisson system with strong magnetic field, Commun. in Comput. Phys. 18 (2015), pp. 263-296.

\bibitem{ICM}{\sc L. Gauckler, E. Hairer, Ch. Lubich}, Dynamics, numerical analysis, and some geometry, Proc. Int. Cong. Math. 1 (2018), pp. 453-486.

\bibitem{Hairer000}{\sc E. Hairer}, Energy-preserving variant of collocation methods, JNAIAM J. Numer. Anal. Ind. Appl. Math. 5 (2010), pp. 73-84.


\bibitem{Hairer00}{\sc E. Hairer, Ch. Lubich}, Long-time energy conservation of
numerical methods for oscillatory differential equations, SIAM J.
Numer. Anal. 38 (2000), pp. 414-441.

\bibitem{Hairer16}{\sc E. Hairer, Ch. Lubich}, Long-term analysis of the
St\"{o}rmer-Verlet method for Hamiltonian systems with a
solution-dependent high frequency, Numer. Math. 134 (2016), pp. 119-138.

\bibitem {Hairer2017-1}{\sc E. Hairer, Ch. Lubich}, Energy behaviour of the Boris method for
charged-particle dynamics,  BIT 58  (2018),  pp. 969-979.

\bibitem {Hairer2017-2}{\sc E. Hairer, Ch. Lubich}, Symmetric multistep methods for
charged-particle dynamics,  SMAI J. Comput. Math. 3 (2017), pp. 205-218.

\bibitem {Hairer2018}{\sc E. Hairer, Ch. Lubich}, Long-term analysis of a variational integrator for
charged-particle dynamics in a strong magnetic field,
 Numer. Math. 144 (2020), pp. 699-728.

\bibitem {hairer2006} {\sc E. Hairer, Ch. Lubich, G. Wanner}, Geometric Numerical
Integration: Structure-Preserving Algorithms for Ordinary
Differential Equations, 2nd edn.  Springer-Verlag, Berlin,
Heidelberg, 2006.

  \bibitem{lubich19}{\sc E. Hairer, Ch. Lubich,  B. Wang}, A filtered Boris algorithm for
charged-particle dynamics in a strong magnetic field,
Numer. Math. 144 (2020), pp. 787-809.

\bibitem {He2015}{\sc Y. He,   Y. Sun, J. Liu, H. Qin}, Volume-preserving algorithms for charged particle dynamics,
J. Comput. Phys. 281 (2015), pp. 135-147.

\bibitem {He2017}{\sc Y. He, Z. Zhou,
Y. Sun, J. Liu, H. Qin}, Explicit K-symplectic algorithms for charged
particle dynamics, Phys. Lett. A 381 (2017), pp. 568-573.

%\bibitem {Hochbruck2010} {\sc M. Hochbruck, A.  Ostermann}, Exponential
%integrators, Acta Numer.  19 (2010), pp. 209-286.

\bibitem {Ostermann15}{\sc C. Knapp, A. Kendl,  A. Koskela, A. Ostermann}, Splitting methods for time integration of trajectories in combined electric and magnetic fields, Phys. Rev. E  92 (2015), pp. 063310.

\bibitem {VP7} {\sc M. Kraus, K. Kormann, P. Morrison, E. Sonnendr\"ucker}, GEMPIC: geometric electromagnetic Particle In Cell methods, Journal of Plasma Physics 4 (2017), pp. 83.

\bibitem {add1}
{\sc W.W. Lee}, Gyrokinetic approach in particle simulation, Phys. Fluids 26 (1983).

\bibitem{Li-ANM}{\sc T. Li, B. Wang}, Efficient energy-preserving methods for charged-particle dynamics, Appl. Math. Comput. 361 (2019), pp. 703-714.

\bibitem{Li-AML}{\sc T. Li, B. Wang}, Arbitrary-order energy-preserving methods for charged-particle dynamics, Appl. Math. Lett. 100 (2020), pp. 106050.

%\bibitem{Li-WuSISC2016}{\sc Y. Li, X. Wu}, Exponential
%integrators preserving first integrals or Lyapunov functions for
%conservative or dissipative systems, SIAM J. SCI. Comput. 38 (2016),
%pp. A1876-A1895.

  \bibitem{Splitting}
{\sc R.I. McLachlan, G.R.W. Quispel}, Splitting methods, Acta Numer. 11 (2002), pp. 341-434.

\bibitem{AVF} {\sc R.I. McLachlan, G.R.W. Quispel, N. Robidoux}, Geometric integration using discrete gradients, Philos. Trans. R. Soc. A 357 (1999), pp. 1021-1046.


\bibitem {Northrop63} {\sc T.G. Northrop}, The adiabatic motion of charged particles.
Interscience Tracts on Physics and Astronomy, Vol. 21. Interscience
Publishers John Wiley  and Sons New York-London-Sydney, 1963.

\bibitem{scaling2}
{\sc S. Possanner}, Gyrokinetics from variational averaging: existence and error bounds, J. Math. Phys.
59 (2018), pp. 082702.

\bibitem{PRL1}
{\sc H. Qin, X. Guan}, Variational symplectic integrator for long-time simulations of the guiding-center motion of charged particles in general magnetic fields, Phys. Rev. Lett. 100 (2008), pp. 035006.

\bibitem {Qin2013}{\sc H. Qin, S. Zhang, J. Xiao, J. Liu, Y. Sun, W. Tang}, Why is
Boris algorithm so good?, Phys. Plasmas 20  (2013), pp. 084503.

\bibitem{Quispel2008} {\sc G.R.W. Quispel, D.I. McLaren}, A new class of energy-preserving
numerical integration methods, J. Phys. A: Math. Theor.  41 (2008),  pp. 045206.

\bibitem{SonnendruckerBook}
{\sc E. Sonnendr\"{u}cker}, {Numerical Methods for Vlasov Equations}, Lecture notes, 2016.

 \bibitem {Tao2016}{\sc M. Tao}, Explicit high-order symplectic
integrators for charged particles in general electromagnetic fields,
J. Comput. Phys.  327 (2016), pp. 245-251.

\bibitem{Wang2020}{\sc  B. Wang}, Exponential energy-preserving methods for charged-particle dynamics in a  strong and constant magnetic field, to appear on J. Comput. Appl. Math. (2020).

 \bibitem {Webb2014}{\sc S.D. Webb}, Symplectic integration of
magnetic systems, J. Comput. Phys. 270 (2014), pp. 570-576.

\bibitem {Zhang2016} {\sc R. Zhang, H.
Qin, Y. Tang, J. Liu, Y. He,  J. Xiao}, Explicit symplectic
algorithms based on generating functions for charged particle
dynamics, Phys. Rev. E 94 (2016), pp. 013205.


\end{thebibliography}
\end{document}